\documentclass[12pt]{article}

\usepackage{amsmath}
\usepackage{amssymb}
\usepackage{amsthm}
\usepackage{latexsym}
\usepackage{color}
\usepackage{graphicx}
\usepackage{color}

\DeclareSymbolFont{calletters}{OMS}{cmsy}{m}{n}
\DeclareSymbolFontAlphabet{\mathcal}{calletters}

\def\be{\begin{eqnarray}}
\def\ee{\end{eqnarray}}

\def\b*{\begin{eqnarray*}}
\def\e*{\end{eqnarray*}}

\newcommand{\nn}{\nonumber}
\newcommand{\bq}{\begin{equation*}}
\newcommand{\eq}{\end{equation*}}		
\newcommand{\bqn}{\begin{equation}}
\newcommand{\eqn}{\end{equation}}
\newcommand{\bqq}{\begin{eqnarray*}}
\newcommand{\eqq}{\end{eqnarray*}}
\newcommand{\bqqn}{\begin{eqnarray}}
\newcommand{\eqqn}{\end{eqnarray}}

\newtheorem{Theorem}{Theorem}[section]

\newtheorem{Proposition}[Theorem]{Proposition}

\newtheorem{Assumption}[Theorem]{Assumption}

\newtheorem{Lemma}[Theorem]{Lemma}
\newtheorem{Corollary}[Theorem]{Corollary}
\newtheorem{Remark}[Theorem]{Remark}

\makeatletter \@addtoreset{equation}{section}

\usepackage{color}
\usepackage{xcolor}

\def \E{\mathbb{E}}
\def \F{\mathbb{F}}
\def \L{\mathbb{L}}
\def \P{\mathbb{P}}
\def \Q{\mathbb{Q}}
\def \R{\mathbb{R}}

\def \N{\mathbb{N}}
\def \V{\mathbb{V}}

\def\Ac{{\cal A}}
\def\Bc{{\cal B}}
\def\Cc{{\cal C}}
\def\Dc{{\cal D}}

\def\Fc{{\cal F}}

\def\Hc{{\cal H}}

\def\Mc{{\cal M}}

\def\Pc{{\cal P}}

\def\Vc{{\cal V}}

\def\Omt{\tilde{\Om}}
\def\Fct{\tilde{\Fc}}
\def\Pt{\tilde{\P}}

\def \Om{\Omega}
\def \om{\omega}

\def \Omb{\overline{\Om}}

\def \eps{\varepsilon}
\def \xb{\mathbf{x}}

\def \0{\mathbf{0}}

\def \Fcb{\overline{{\cal F}}}
\def \Fbb{\overline{\mathbb{F}}}
\def \Ft{\tilde{\F}}

\newcommand{\rmi}{{\rm (i)$\>\>$}}

\newcommand{\rmii}{{\rm (ii)$\>\>$}}
\newcommand{\rmiii}{{\rm (iii)$\>\,    \,$}}
\newcommand{\rmiv}{{\rm (iv)$\>\>$}}

\def\x{\times}

\def\1{{\bf 1}}
\def \proof{{\noindent \bf Proof. }}

\def \lambdab{\bar{\lambda}}

\title{Optimal Skorokhod embedding given full marginals and Az\'ema-Yor peacocks
\footnote{The authors gratefully acknowledge the financial support of the ERC 321111 Rofirm, the ANR Isotace, the Chairs Financial Risks (Risk Foundation, sponsored by Soci\'et\'e G\'en\'erale), Finance and Sustainable Development (IEF sponsored by EDF and CA).}
}

\author{Sigrid K\"allblad\thanks{CMAP, Ecole Polytechnique, France. sigrid.kallblad@cmap.polytechnique.fr}
	\and Xiaolu TAN\thanks{CEREMADE, University of Paris-Dauphine, France.
	tan@ceremade.dauphine.fr}
	\and Nizar Touzi\thanks{CMAP, Ecole Polytechnique, France. 
	nizar.touzi@polytechnique.edu}
}
\date{\today}

\begin{document}

\maketitle

\abstract{We consider the optimal Skorokhod embedding problem (SEP) given full marginals 
	over the time interval $[0,1]$.
	The problem is related to the study of extremal martingales associated with a peacock 
	(``process increasing in convex order'', by Hirsch, Profeta, Roynette and Yor \cite{HPRY}).
	A general duality result is obtained by convergence techniques.
	We then study the case where the reward function depends on the maximum of the embedding process,
	which is the limit of the martingale transport problem studied in Henry-Labord\`ere, Ob{\l}{\'o}j, Spoida and Touzi \cite{HOST}.
	Under technical conditions, some explicit characteristics of the solutions to the optimal SEP as well as to its dual problem are obtained.
	We also discuss the associated martingale inequality.	

	\vspace{2mm}

	{\bf Key words.} Skorokhod embedding problem, peacocks, martingale inequality, 
	martingale transport problem, maximum of martingale given marginals

}

\section{Introduction} 
\label{sec:introduction}

	For a given probability measure $\mu$ on $\R$, centered and with finite first moment, 
	the Skorokhod embedding problem (SEP) consists in finding a stopping time $T$ for a Brownian motion $W$, such that $W_{T} \sim \mu$ 
	and the stopped process $(W_{T \wedge \cdot})$ is uniformly integrable.
	We consider here an extended version.
	Let $(\mu_t)_{t \in [0,1]}$ be a family of probability measures
	that are all centered, have finite first moments, and are non-decreasing in convex order,
	i.e. $t \mapsto \mu_t(\phi) := \int_{\R} \phi(x) \mu_t(dx)$ is non-decreasing for every convex function $\phi: \R\to \R$.
	The extended Skorokhod embedding problem is to find a non-decreasing family of stopping times, 
	$(T_t)_{t \in [0,1]}$, for a Brownian motion $W$,
	such that $W_{T_t} \sim \mu_t, \forall t \in [0,1]$,
	and each stopped process $(W_{T_1 \wedge \cdot})$ is uniformly integrable.
	We study here an optimal Skorokhod embedding problem
	which consists in maximizing a reward value among the class of all such extended embeddings.

	For such a family $\mu = (\mu_t)_{t \in [0,1]}$, it follows from Kellerer's theorem (see e.g. Kellerer \cite{Kellerer} or Hirsch and Roynette \cite{HR}) that there exists at least one (Markov) martingale whose one-dimensional marginal distributions coincide with $\mu$.
	Assume in addition that $t \mapsto \mu_t$ is right-continuous,
	then any associated martingale admits a right-continuous modification.
	Moreover, Monroe \cite{Monroe} showed that any right-continuous martingale can be embedded into a Brownian motion with a non-decreasing family of stopping times.
	This implies that the collection of solutions to the extended Skorokhod embedding problem is non-empty.
	Furthermore, the above optimal SEP is thus related to the study of extremal martingales associated with peacocks.
	A peacock (or PCOC ``Processus Croissant pour l'Ordre Convexe'') is a continuous time stochastic process whose one-dimensional marginal distributions are non-decreasing in convex order
	according to Hirsch, Profeta, Roynette and Yor \cite{HPRY}.	
	Since Kellerer's theorem ensures the existence of martingales with given one-dimensional marginal distributions,
	the interesting subject is to construct these associated martingales;
	we refer to the book \cite{HPRY} and the references therein for various techniques.
	We also mention that when the marginal distributions are those of a Brownian motion,
	an associated martingale is also referred to as a fake Brownian motion; see e.g. \cite{alb, fan, hob, ole}.

	Our problem to find the extremal martingales associated with a given peacock
	is motivated by its application in finance.
	Specifically, given the prices of vanilla options for all strikes at a maturity, one can recover the marginal distribution of the underlying at this maturity (see e.g. Breeden and Litzenberger \cite{BL}).
	Taking into consideration all martingales fitting these marginal distributions, one then obtains model-independent bounds on arbitrage-free prices of exotic options.
	The problem was initially studied using the SEP approach by Hobson \cite{Hobson1998} and many others.
	This approach is based on the fact that any continuous martingale can be viewed as a time-changed Brownian motion;
	let us refer to the survey paper of Ob{\l}{\'o}j \cite{Obloj} and Hobson \cite{Hobson2011}.
	More recently, it has also been studied using the so-called martingale transport approach introduced in Beiglb\"ock, Henry-Labord\`ere and Penkner \cite{BHP} and Galichon, Henry-Labord\`ere and Touzi \cite{GHT}. Since then, there has been an intensive development of the literature on martingale optimal transport and the connection with model-free hedging in finance. In the present context of full marginals constraint, Henry-Labord\`ere, Tan and Touzi \cite{HTT} considered reward functions satisfying the so-called martingale Spence-Mirrlees condition, and solved an example of martingale transport problem with quasi-explicit construction of the corresponding martingale peacock and the optimal semi-static hedging strategy.

	In this paper, we study extremal martingale peacocks using the optimal SEP approach.
	First, taking the limit of a duality result for a general optimal SEP under finitely-many marginal constraints, 
	established in Guo, Tan and Touzi \cite{GTT} (extending a duality result in Beiglb\"ock, Cox and Huesmann \cite{BH}),
	we obtain a general duality result for the optimal SEP under full marginal constraints.
	Thereafter, we study the case where the reward function depends on the realized maximum of the embedding process.
	For the problem with finitely-many marginal constraints,
	the optimal embedding is then given by the iterated Az\'ema-Yor embedding proposed by
	Ob{\l}{\'o}j and Spoida \cite{OS}, which extends the embeddings of Az\'ema and Yor \cite{AzemaYor} and Brown, Hobson and Rogers \cite{BHR}.
	The solution to the associated dual problem, as well as the optimal value, is studied in Henry-Labord\`ere, Ob{\l}{\'o}j, Spoida and Touzi \cite{HOST}.
	By applying limiting arguments, we obtain some explicit characterization of the optimal value and the primal and dual optimizers for the corresponding optimal SEP under full marginal constraints.

	\vspace{1mm}

	The rest of the paper is organized as follows. The main results are presented in Section \ref{sec:main_results}: 
	in Section \ref{subsec:OSEP} we formulate our optimal SEP given full marginals,
	in Section \ref{subsec:duality} we provide the general duality result, 
	in Section \ref{subsec:MMfullM} we focus on a class of maximal reward functions for which we specify the value of the problem and give the explicit form of a dual optimizer, 
	and in Section \ref{subsec:MartIneq} we present an associated martingale inequality. 
	In Section \ref{sec:discussions} we provide further discussion of our results and relate them to the finite-marginal SEP. 
	Specifically, we show that our full marginal optimal SEP has the interpretation as the limit of certain optimal SEP and martingale transport problems under finitely many marginals. 
	The proofs are completed in Section \ref{sec:Proofs}.

	\vspace{2mm}
	
	\noindent {\bf Notation.}
	\rmi Let $\Om := C(\R_+, \R)$ denote the canonical space of all continuous paths $\om$ on $\R_+$ such that $\om_0 = 0$, 
	$B$ be the canonical process, $\P_0$ be the Wiener measure under which $B$ is a standard Brownian motion,
	$\F^0 = (\Fc^0_t)_{t \ge 0}$ denote the canonical filtration generated by $B$,
	and $\F^a = (\Fc^a_t)_{t \ge 0}$ be the augmented filtration under $\P_0$.
	
	We equip $\Om$ with the compact convergence topology 
	(see e.g. Whitt \cite{whitt} or Stroock and Varadhan \cite{SV}):
	\be \label{eq:rho}
		\rho(\om, \om ')
		&:=&
		\sum_{n \ge 1} \frac{1}{2^n} \frac{\sup_{0 \le t \le n} \big| \om_t - \om_t' \big|}{1 + \sup_{0 \le t \le n} \big| \om_t - \om_t' \big|},
		~~~ \forall \om , \om' \in \Om.
	\ee
	Then $(\Om, \rho)$ is a Polish space (separable and complete metric space).

	\noindent \rmii Let $\V^+_r = \V^+_r([0,1], \R_+)$ denote the space of all non-decreasing c\`adl\`ag functions on $[0,1]$ taking values in $\R_+$.
	Similarly, let $\V^+_l=\V^+_l([0,1],\R)$ denote the space of all non-decreasing c\`agl\`ad functions on $[0,1]$ taking values in $\R$. 
	
	Further, we equip $\V^+_r$ and $\V^+_l$ with the L\'evy metric: for all $\theta, \theta' \in \V^+_r$, 
	\bqn \label{eq:levy_metric}
		d(\theta,\theta')
		~:=~
		\inf \Big\{\eps>0 ~: 
		\theta_{t-\eps} - \eps
		~\le~ 
		\theta'_t
		~\le~
		\theta_{t+\eps}+\eps,
		~\forall t \in[0 ,1]\Big\},
	\eqn
	where we extend the definition of $\theta$ to $[- \eps, 1+\eps]$ by letting
	$\theta_s := \theta_0$ for $s \in [-\eps, 0]$ and $\theta_s := \theta_1$ for $s \in [1, 1+\eps]$.
	Then $\V^+_r$ and $\V^+_l$ are both Polish spaces.

	\noindent \rmiii
	As in El Karoui and Tan \cite{EKT1, EKT2}, 
	we define an enlarged canonical space by $\Omb := \Om \x \V^+_r$,
	where the canonical process is denoted by $\overline B = (B, T)$.
	The canonical filtration on the enlarged canonical space is denoted by $\Fbb = (\Fcb_t)_{t \ge 0}$,
	where $\Fcb_t$ is generated by $(B_s)_{s \in [0,t]}$ and all the sets $\{ T_r \le s \}$ 
	for $s \in [0,t]$ and $r \in [0,1]$.
	In particular, all the canonical variables $(T_r)_{r \in [0,1]}$ are $\Fbb$-stopping times.

	We notice that 
	the $\sigma$-field $\Fcb_{\infty} := \bigvee_{t \ge 0} \Fcb_t$ coincides with the Borel $\sigma$-field of the Polish space $\Omb$ (see Lemma \ref{lemm:BorelTribu}).
	
	For a set $\Pc$ of probability measures on $\Omb$, we say that a property holds $\Pc-$quasi-surely (q.s.) if it holds $\P-$a.s. for all $\P \in \Pc$.
	
	\vspace{2mm}

	\noindent \rmiv Let $\Cc_b$ denote the space of all bounded continuous functions from $\R$ to $\R$, and by $\Cc_1$ the space of all functions $f:\R \to \R$ such that $\frac{f(x)}{1 + |x|} \in \Cc_b$.

\section{Main results}
\label{sec:main_results}

	Throughout the paper, we are given a family of probability measures on $\R$, $\mu = (\mu_t)_{t \in [0,1]}$, satisfying the following condition.

	\begin{Assumption} \label{ass:mu}
		The family of marginal distributions, $\mu = (\mu_t)_{t \in [0,1]}$, satisfies: 
		\bq
			\int_\R~|x|~\mu_t(dx)<\infty
			~~\textrm{and}~~
			\int_\R ~x~ \mu_t(dx)=0,
			~~\textrm{$t\in[0,1]$}.
		\eq	
		Furthermore, $\mu_0 = \delta_{\{0\}}$, $t \mapsto \mu_t$ is c\`adl\`ag w.r.t. the weak convergence topology,
		and $\mu$ is non-decreasing in convex ordering, i.e. for every convex function $\phi:\R\to\R$,
		$$
			\mu_s(\phi)
			~~\le~~
			\mu_t(\phi) ~:=~ \int_\R\phi(x) \mu_t(dx),
			~~~~\mbox{whenever}~ s \le t.
		$$ 
	\end{Assumption}

\subsection{The optimal SEP given full marginals}
\label{subsec:OSEP}
	
	First, let us formulate the optimal SEP given full marginals.
	Let $\Pc(\Omb)$ denote the collection of all Borel probability measures on the canonical space $\Omb$,	and define
	\be
		\Pc
		&:=&
		\Big\{
			\P\in\Pc(\Omb) ~: 
			B ~\mbox{is a}~ \Fbb \mbox{-Brownian motion and}~\nonumber\\
		&&~~~~~~~~~~~~~~~~~~~~
			B_{\cdot \wedge T_1}
		~\mbox{is uniformly integrable under}
		~\P
		\Big\}.
	\ee
	For the given marginals $\mu = (\mu_t)_{t\in[0,1]}$, we then define
	\b*
		\Pc(\mu)
		&:=&
		\Big\{
			\P \in \Pc ~: 
			B_{T_t} \sim^{\P} \mu_t,
			~\forall t \in [0,1]
		\Big\}.
	\e*
	
	\begin{Lemma} \label{lemm:canon_form}
		Suppose that Assumption \ref{ass:mu} holds true,
		then $\Pc(\mu)$ is non-empty.
	\end{Lemma}
	
	\begin{proof}
	Since the marginal distributions $\mu = (\mu_t)_{t \in [0,1]}$ satisfy Assumption \ref{ass:mu}, 
	it follows from Kellerer's theorem (see e.g. Kellerer \cite{Kellerer} or Hirsch and Roynette \cite{HR}) that there is a martingale $M$ such that $M_t \sim \mu_t$ for all $t \in [0,1]$.
	Since $t \mapsto \mu_t$ is right-continuous, the martingale $M$ can be chosen to be right-continuous.
	It follows from Theorem 11 in Monroe \cite{Monroe}, that there is a Brownian motion $W$ and a family of non-decreasing and right-continuous stopping times $(\tau_t)_{t \in [0,1]}$, such that $(W_{\tau_1 \wedge \cdot})$ is uniformly integrable and $(W_{\tau_{\cdot}})$ has the same finite-dimensional distributions as $(M_{\cdot})$.
	In consequence, the probability induced by $(W_{\cdot}, \tau_{\cdot})$ on $\Omb$ belongs to $\Pc(\mu)$.
	\qed
	\end{proof}

	
	The main objective of the paper is to study the following optimal Skorokhod Embedding Problem (SEP) under full marginal constraints:
	\be \label{eq:main_problem}
		P(\mu)
		~=~
		\sup_{\P\in\Pc(\mu)}
		\E^{\P} \big[ \Phi\big( B_{\cdot}, T_{\cdot} \big) \big],
	\ee
	where  $\Phi: \Omb \to \R$ is a reward function which is assumed to be upper semicontinuous and bounded from above.

	\vspace{3mm}
	
	The optimal SEP \eqref{eq:main_problem} under full marginal constraints is given in a weak formulation. 
	We specify this next. For given marginals $\mu = (\mu_t)_{t \in [0,1]}$, let a $\mu$-embedding be a term
	\be \label{eq:alpha}
		\alpha 
		&=& 
		\big(\Om^{\alpha}, \Fc^{\alpha}, \P^{\alpha}, \F^{\alpha} = (\Fc_t^{\alpha})_{t \ge 0}, (W^{\alpha}_t)_{t \ge 0}, (T^{\alpha}_s)_{s\in [0,1]} \big),
	\ee
	such that in the filtered space $\big(\Om^{\alpha}, \Fc^{\alpha}, \P^{\alpha}, \F^{\alpha} \big)$,
	$W^{\alpha}_{\cdot}$ is a Brownian motion, 
	$T^{\alpha}_{\cdot}$ is a non-decreasing c\`adl\`ag family of stopping times,
	the stopped process $(W^{\alpha}_{T^{\alpha}_1 \wedge \cdot})$ is uniformly integrable,
	and $W^{\alpha}_{T^{\alpha}_t} \sim^{\P^{\alpha}} \mu_t$ for every $t \in [0,1]$.
	Denote by $\Ac(\mu)$ the collection of all $\mu$-embeddings $\alpha$.	
	It is clear that every $\mu$-embedding  $\alpha \in \Ac(\mu)$ induces on the canonical space $\Omb$ a probability measures $\P \in \Pc(\mu)$.
	Further, every $\P \in \Pc(\mu)$ together with the canonical space $\Omb$ and canonical process $\overline B$, forms a stopping term in $\Ac(\mu)$.
	It follows that the set $\Pc(\mu)$ is the collection of all probability measures $\P$ on $\Omb$, induced by the embeddings $\alpha \in \Ac(\mu)$.
	As a direct consequence, the optimal SEP \eqref{eq:main_problem}
	admits the following equivalent formulation: 
	\be \label{eq:P_embedding}
		P(\mu)
		&:=&
		\sup_{\alpha \in \Ac(\mu)}
		\E^{\alpha} \big[ \Phi \big( W^{\alpha}_{\cdot}, T^{\alpha}_{\cdot} \big) \big].
	\ee

\subsection{Duality for the full marginal SEP problem}
\label{subsec:duality}

	In order to introduce the dual problem,
	let $\L^2_{loc}$ denote the space of all $\Fbb$-progressively measurable processes, $H  = (H_t)_{t \ge 0}$, defined on the enlarged canonical space $\Omb$ and such that
	\b*
		\int_0^t H_s^2 ds ~<~ \infty, ~~\P\mbox{-a.s., for all}~ t > 0 ~\mbox{and}~ \P \in \Pc.
	\e*
	For every $H \in \L^2_{loc}$ and $\P \in \Pc$, the stochastic integral of $H$ w.r.t. the canonical process $B$ under $\P$, denoted by $(H \cdot B)_{\cdot}$, is well-defined. 
	An adapted process $M = (M_t)_{t \ge 0}$ defined on $\Omb$ is called a strong supermartingale under $\P$, if $M_\tau$ is integrable for all $\Fbb$-stopping times $\tau\ge 0$, and for $\Fbb$-stopping times $\tau_1 \le \tau_2$ we have that
	$\E^{\P} \left[ M_{\tau_2}| \Fcb_{\tau_1} \right]~\le ~M_{\tau_1}$.
	We then define $\Hc$ by
	\b*
		\Hc 
		:= 
		\Big\{ H  \in \L^2_{loc} :(H\cdot B)_{\cdot}
			~\mbox{is a strong supermartingale under every $\P\in\Pc$}
		\Big\}.
	\e*
	Next, let $M([0,1])$ denote the space of all finite signed measures on $[0,1]$. 
	Note that $M([0,1])$ is a Polish space under the weak convergence topology.
	Further, let $\Lambda$ denote the space of all $\lambda:\R\to M([0,1])$ admitting the representation
	\b*
		\lambda(x, dt) ~=~ \lambda_0(x,t) \lambdab(dt),
	\e* 
	for some finite positive measure $\lambdab \in M([0,1])$
	and some locally bounded measurable function $\lambda_0: \R \x [0,1] \to \R$.
	For $\mu = (\mu_t)_{t \in [0,1]}$, we define
	\b*
		\Lambda(\mu)
		&:=&
		\Big\{ 
			\lambda \in \Lambda ~: 
			\mu(|\lambda|) := \int_0^1 \int_{\R} \big| \lambda_0 (x,t) \big| \mu_t(dx)\bar\lambda(dt) ~<~\infty
		\Big\},
	\e*
	and
	\bqn \label{eq:quest:price}
		\mu(\lambda)~:=~
		\int_0^1 \lambda(x,dt)\mu_t(dx)~=~\int_0^1\int_\R\lambda_0(x,t)\mu_t(dx)\bar\lambda(dt),
		~~\forall \lambda \in \Lambda(\mu).
	\eqn
	With the notation $\lambda(\overline B):=\int_0^1\lambda_0(B_{T_s}, s)\bar\lambda(ds)$, we finally set
	\bqn \label{eq:Dc_mu}
		\Dc(\mu) ~:=~ \Big\{ (\lambda, H) \in \Lambda(\mu) \x \Hc ~: 
			\lambda\left(\overline B\right) + \big(H\cdot B\big)_{T_1}
			~\ge~
			\Phi\left(\overline B_{\cdot}\right),
			~\Pc\mbox{-q.s.}
		\Big\}.
	\eqn
	The dual problem for the optimal SEP \eqref{eq:main_problem}, under full marginal constraints, is then defined as follows:
	\be \label{eq:dual:sep}
		D(\mu)
		&:=&
		\inf_{(\lambda, H) \in \Dc(\mu)}
		 	\mu(\lambda).
	\ee
	
	Our first main result is the following.
	
	\begin{Theorem} \label{theo:main}
		Let Assumption \ref{ass:mu} hold true. 
		Suppose in addition that $\Phi:\Omb\to\R$ is upper semicontinuous, bounded from above
		and satisfies $\Phi(\om, \theta) = \Phi(\om_{\theta_1 \wedge \cdot}, \theta)$ for all $(\om,\theta) \in \Omb$.
		Then, there exists a solution $\widehat \P \in \Pc(\mu)$ to the problem $P(\mu)$ in \eqref{eq:main_problem}, and
		we have the duality
			\b*
				\E^{\widehat{\P}} \Big[\Phi \big( B_{\cdot}, T_{\cdot} \big) \Big] 
				~~=~~
				P(\mu) 
				&=& D(\mu).
			\e* 			
	\end{Theorem}
	
	We also introduce the weaker version of the dual problem:

	\be \label{eq:dual:sep2}
			D_0(\mu)
			&:=&
			\inf_{(\lambda, H) \in \Dc_0(\mu)}
		 	\mu(\lambda),
	\ee
	with $\Dc_0(\mu)$ given by 
	\bq \label{eq:dual_strat_set0}
		\Dc_0(\mu) := \Big\{ (\lambda, H) \in \Lambda(\mu) \x \Hc ~: 
			\lambda\left(\overline B\right) + \big(H\cdot B\big)_{T_1}
			\ge
			\Phi\left(\overline B_{\cdot}\right),
			~\Pc(\mu)\mbox{-q.s.}
		\Big\}.~~~
	\eq
	
	As a consequence of Theorem \ref{theo:main}, we have the following result. 
	
	\begin{Corollary} \label{cor:main}			
		Under the same conditions as in Theorem \ref{theo:main}, it holds that
		\b*
			P(\mu) ~=~ D_0(\mu) ~=~ D(\mu).
		\e* 			
	\end{Corollary}
	
	\begin{proof}
		Let $(\lambda, H) \in \Dc_0(\mu)$. For any $\P \in \Pc(\mu)$, taking expectation over the inequality in the definition of $\Dc_0(\mu)$, one obtains $\mu(\lambda) \ge \E^{\P}\big[ \Phi(B_{\cdot}, T_{\cdot}) \big]$. Hence, $\mu(\lambda) \ge P(\mu)$, which yields the weak duality $ D_0(\mu) \ge P(\mu)$. 
		Since $\Pc(\mu)\subseteq \Pc$, it follows that $D_0(\mu)\le D(\mu)$. In consequence, the result follows from Theorem \ref{theo:main}.
	\qed
	\end{proof}

\subsection{Maximum maximum given full marginals}
\label{subsec:MMfullM}

	In this subsection, we restrict to the case where
	\be \label{eq:Phi_phi}
		\Phi(\om,\theta) &=& \phi \big(\om^*_{\theta_1} \big),
		~~\mbox{with}~
		\om^*_{t} := \max_{0 \le s \le t} \om_s,
		~~t\ge 0,
	\ee
	for some bounded, non-decreasing and upper semi-continuous (or equivalently c\`adl\`ag) function $\phi: \R_+ \to \R$.
	According to Theorem \ref{theo:main} and Lemma \ref{lem:cont} below, we have the duality $P(\mu) = D(\mu)$.
	Our main concern in this part is to compute this optimal value and, in turn, find and characterize the solution to the dual problem \eqref{eq:dual:sep2}.

	First, we introduce some further conditions on the marginals $\mu$.
	To this end, let $c(t,x) := \int_\R (y-x)^+ \mu_t(dy)$ for every $(t,x) \in [0,1] \x \R$.
	\begin{Assumption}\label{ass:A}
		\rmi The function $c$ is differentiable in $t$ and the derivative function $\partial_t c$ is continuous, i.e. $\partial_t c(t,x) \in C([0,1] \x \R)$.
		
		\vspace{1mm}

		\noindent \rmii
		There exists a sequence of discrete time grids $(\pi_n)_{n \ge 1}$ with
		$\pi_n = (0 =t^n_0 < t^n_1 < \cdots <t^n_n = 1)$,
		such that $|\pi_n| \to 0$,
		and, for all $n\ge 1$, the family of finite marginals 
		$(\mu_{t^n_i})_{i=1}^n$ satisfies Assumption $\circledast$ in \cite{OS}. 
	\end{Assumption}

	We introduce a minimization problem for every fixed $m \ge 0$. With the convention that $\frac{0}{0} = 0$ and $\frac{c}{0} = \infty$ for $c >0$, let
	\be \label{eq:maxzeta}
		C(m)
		&:=&
		\inf_{\zeta \in \V^+_l:~\zeta\;\le\; m}~
		\left\{~ \frac{c(0, \zeta_0)}{m - \zeta_0}
			~+~
			\int_0^1 \frac{\partial_t c(s, \zeta_s)}{m - \zeta_s} ds
		~\right\}.
	\ee
	Our first result is on the value of the optimal SEP \eqref{eq:main_problem}.
	\begin{Theorem} \label{them:main2}
		Let $\Phi$ be given by \eqref{eq:Phi_phi} for some bounded, 
		non-decreasing and c\`adl\`ag function $\phi$.
		Suppose that Assumptions \ref{ass:mu} and \ref{ass:A} \rmi hold true.
		Then 
		\be \label{eq:OptValue_Ineq}
			P(\mu)
			~=~
			D(\mu)
			&\le&
			\phi(0)+\int_0^\infty C(m)d\phi(m).
		\ee
		Suppose in addition that Assumption \ref{ass:A} \rmii holds true, then equality holds in \eqref{eq:OptValue_Ineq}.	
	\end{Theorem}

	Our second result is on the existence and characterization of a specific dual optimizer.	
	To this end, for $m>0$ and $\zeta\in\V^+_l$ such that $\zeta(1)<m$, let the associated functions $\lambda^{\zeta,m}_c$ and $\lambda^{\zeta,m}_d$ be given by
	\b*
		\lambda^{\zeta,m}_c (x,t) 
		&:=&
		\frac{m-x}{(m-\zeta_t)^2}
		\1_{\{x\ge \zeta_t\}}
		\1_{D_m^c}(t),
	\e*
	and
	\b*
		\lambda^{\zeta,m}_d (x,t)
		&:=&
		\frac{1}{\Delta \zeta_t}\left(\frac{(x-\zeta_t )^+}{m-\zeta_t }-\frac{(x-\zeta_{t+})^+}{m-\zeta_{t+}}\right)
		\1_{D_m}(t)
		~+~
		\frac{(x-\zeta_1 )^+}{m-\zeta_1} \1_{\{t = 1\}}
		,
	\e* 
	where $\Delta \zeta_t := \zeta_{t+} - \zeta_t$, 
	$D_m := \{t\in[0,1): \Delta \zeta_t > 0 \}$
	and $D_m^c := [0,1) \setminus D_m$. We then define
	\bqn \label{eq:lambda_m}
		\lambda^{\zeta,m} (x,dt)
		~:=~\left(
		\lambda^{\zeta,m}_c (x,t)
		+\lambda^{\zeta,m}_d (x,t)
		\right) d\zeta_t.
	\eqn 
	It is clear that $\lambda^{\zeta,m} \in \Lambda$.
	Next, we define the dynamic term. Let $\tau_m(\bar\omega)=\inf\{t\ge 0:\omega_{\theta(t)}\ge m\}$. Further, let $\theta^{-1}:\R_+\to[0,1]$ the right continuous inverse function of $s \mapsto \theta_s$, given by 
	$\theta^{-1}_s:= \sup\{r\in[0,1]:\theta_r\le s\}$. 
	We note that $\theta^{-1}_s(\bar\omega)$ is $\mathcal{\overline F}_s$-measurable for fixed $s$ 
	and, thus, it is c\`adl\`ag and $\Fbb$-adapted and therefore $\Fbb$-progressively measurable. 
	With $I^-:\R_+\to\R_+$ and $I^+:\R_+\to\R_+$ given, respectively, by $I^-(s):=\theta(\theta^{-1}(s)-)$ and $I^+(s):=\theta(\theta^{-1}(s))$, we let
		\bqn\label{eq:dyn_strat_n}
			H^{\zeta,m}_s(\omega,\theta)~:=~
			\frac{\1_{\left[\tau_m,I^+(\tau_m)\right]}(s)}{m-\zeta_{\theta^{-1}(\tau_m)}}
			~+~
			\frac{\1_{\big\{m \le \omega^*_{I^-(s)};\zeta_{\theta^{-1}(s)} \le \omega_{I^-(s)}\big\}}}
			{m - \zeta_{\theta^{-1}(s)}}.
		\eqn  
	Finally, let $\zeta:[0,1)\times\R_+\to\R$ such that, for all $m> 0$, $\zeta^m_\cdot\in\V^+_l$ and $\zeta^m_1\le m$, where $\zeta^m_\cdot = \zeta(\cdot,m)$.	
	Assuming that $\int_0^{\infty} \frac{d\phi(m)}{(m - \zeta^m_1)^2}< \infty$, we then define
	\bqn \label{eq:def_lambda}
		\lambda^\zeta(x,dt) := \int_0^\infty \lambda^{\zeta^m,m} (x,dt)d\phi(m)
		~~\mbox{and}~~
		H^\zeta_s := \int_0^\infty H^{\zeta^m,m}_s d\phi(m).
	\eqn

	The construction of the dual optimizer below is based on the existence of a solution to the minimization problem \eqref{eq:maxzeta}.
	\begin{Lemma}\label{lemm:Zeta}
		Let Assumptions \ref{ass:mu} and \ref{ass:A} (i) hold true.
		Then there exists a measurable function $\hat\zeta:[0,1)\times\R_+\to\R$ 
		such that, for all $m> 0$, $\hat\zeta^m_\cdot\in\V^+_l$ is a solution to \eqref{eq:maxzeta}.
	\end{Lemma}

	\begin{Theorem} \label{theo:optimser:dual}
		Suppose that $\phi: \R_+ \to \R$ is non-decreasing and that Assumptions \ref{ass:mu} and \ref{ass:A} hold true.
		Let $\hat\zeta:[0,1)\times\R_+\to\R$ be a measurable function such that, for all $m> 0$, $\hat\zeta^m_\cdot\in\V^+_l$ is a solution to \eqref{eq:maxzeta}, and
		\be	\label{eq:assump_zeta}
			\int_0^{\infty} \frac{d\phi(m)}{(m - \hat \zeta^m_1)^2} ~~<~~ \infty.
		\ee
		Then, $(\hat\lambda,\widehat H):=(\lambda^{\hat\zeta},H^{\hat\zeta}) \in\Lambda(\mu)\times\Hc$. 
		Suppose in addition that $\phi:\R_+ \to \R$ is bounded and continuous and that, for all $t\in[0,1]$,
		\bqn \label{eq:dual_opti_condi}
			\mu_t
			~\mbox{is atomless;}~~
			\hat\zeta^m_t
			~\mbox{and its inverse are both continuous in}~m.
		\eqn
		Then $(\hat\lambda,\widehat H)$ is a dual optimizer for the problem $D_0(\mu)$ in \eqref{eq:dual:sep2}. 
		That is, with $\Phi$ given in \eqref{eq:Phi_phi}, it holds that
		\bqn \label{eq:superhedge}
			\mu(\hat\lambda)=D_0(\mu)
			~~\mbox{and}~~
			\hat\lambda(\overline B)~+~\big(\widehat H\cdot B\big)_{T_1}
			~\ge~
			\Phi(B_{\cdot},T_{\cdot}),\quad \Pc(\mu)\emph{\mbox{-q.s.}}
		\eqn
	\end{Theorem}
	
	\begin{Remark}
		The condition \eqref{eq:dual_opti_condi} 
		is needed to argue the convergence to $(\hat\lambda,\widehat H)$, in an appropriate sense, of the corresponding dual optimizers for the finite marginals case (see Lemma \ref{lemm:conv_static}).
		As seen from the proof, 
		if $\hat\zeta^m_\cdot$ can be represented as a countable sum, i.e.
		\bqn 	\label{eq:assumption_countable}
			\hat\zeta^m_s=\sum_{k=0}^\infty \zeta^m_k\1_{(t_k,t_{k+1}]}(s),
		\eqn
		for some $(\zeta^m_k)_{k \ge 0}$,
		then $(\hat\lambda,\widehat H)$ is a dual optimizer even though condition \eqref{eq:dual_opti_condi} fails. 
	\end{Remark}

\subsection{An associated martingale inequality}
\label{subsec:MartIneq}

	In this section we establish a closely related martingale inequality. We stress that this result does not require Assumption \ref{ass:A}. 
	 
	 \begin{Proposition}\label{prop:mtq:new}
		Let $M$ be a right continuous martingale, $\phi:\R_+\to\R$ non-decreasing and c\`adl\`ag, and $\zeta:[0,1]\times\R_+\to\R$ such that, for $m>0$, $\zeta^m_\cdot\in\V^+_l$ and $\zeta_1^m < m$. 
		Assume further that for each $m>0$,
		\b*
			\Big\{t \in [0,1) ~:x\longmapsto\P\left[M_t \le x\right]\textrm{is discontinuous at}~ x=\zeta^m_t  \Big\}
		\e*
		is a $d\zeta^{m,c}_t$-null set, where $\zeta^{m,c}_t$ is the continuous part of $t \mapsto \zeta^m_t$.
		Then, with $M^*_t := \max_{0 \le s \le t} M_s$, it holds that 
		\b*
			\E \big[ \phi(M^*_1) \big]
			&\le&
			\phi(0)~+
			\int_0^{\infty} 
				\E \left[ \int_0^1  \tilde\lambda^m (M,dt) \right]
			d \phi(m),
		\e*	
		where $\tilde\lambda^m(x,dt):=\lambda^{\zeta^m,m}(x,dt)$ (cf. \eqref{eq:lambda_m}) so that
	{\setlength{\arraycolsep}{-2.8cm}
	\bqq
		\E \left[ \int_0^1  \tilde\lambda^m (M,dt) \right]
		~=~
		\int_0^1 \frac{\P\left[M_t>\zeta^m_t\right](m-\zeta^m_t) - \E\left[\left(M_t-\zeta^m_t\right)^+\right]}{(m-\zeta^m_t)^2}~d\zeta^{m,c}_t
		\\ && ~+~ 
		\frac{\E\left[(M_1 - \zeta^m_1)^+\right]}{m - \zeta^m_1}
		~+~
		\sum_t
		~\bigg[ \frac{\E\left[(M_{t} - \zeta^m_t)^+\right]}{m - \zeta^m_t} - \frac{\E\left[(M_{t} - \zeta^m_{t^+})^+\right]}{m - \zeta^m_{t^+}}\bigg].
	\eqq }
	\end{Proposition}
	
	We conclude this section with a remark on an alternative version of the above martingale inequality.

	\begin{Remark}
		Suppose that $M$ is a c\`adl\`ag martingale such that the function $c_M(t,x) := \E[ (M_t - x)^+]$ is $\mathcal{C}^1$ in $t$. Further, let $\phi:\R_+\to\R$ non-decreasing and c\`adl\`ag, and $\zeta:[0,1]\times\R_+\to\R$ such that, for $m>0$, $\zeta^m_\cdot\in\V^+_l$ and $\zeta_1^m < m$.
		Then,
		\b*
			\E \Big[ \phi( M^*_1 ) \Big]
			&\le&
			\phi(0)
			~+
			\int_0^{\infty}
				\left( \frac{\E [ (M_0 - \zeta^m_0)^+] }{m - \zeta^m_0}
				~+
				\int_0^1 \frac{\partial_t c_M(t,\zeta^m_t) }{m - \zeta^m_t} d t
				\right)
			d \phi(m).
		\e*
		Indeed, due to Monroe \cite{Monroe}, there is $\P\in\Pc(\mu)$ such that 
			\bqn	\label{eq:ineq_max} 
				\E\left[ \phi (M^*_1)\right]
				~=~ \E^{\P} \Big[ \phi \Big( \max_{0 \le t \le 1} B_{T_t} \Big) \Big]
				~\le ~
				\E^{\P} \big[ \phi \big( B^*_{T_1} \big) \big],
			\eqn 
		where the inequality follows as $\phi$ is non-decreasing and $\max_{0 \le t \le 1} B_{T_t}\le B^*_{T_1}$. The above inequality is therefore an immediate consequence of Theorem \ref{them:main2}. 		
	\end{Remark}

\section{Further discussion}
\label{sec:discussions}

	In this section, we provide some discussion on the relation between the optimal SEP and the martingale transport problem, and specify how the optimal SEP given full marginals can be considered 
	as the limit of the approximating problem defined by a finite subset of marginals. We also discuss a numerical scheme for the problem $C(m)$ in \eqref{eq:maxzeta}.

\subsection{The optimal SEP given finitely many marginals}

	First, we consider the optimal SEP given finitely many marginals and recall some results established in previous works.
	To this end, for $n \ge 1$, let $\pi_n = \{ t^n_0, \cdots, t^n_n \}$ be a discrete time grid on $[0,1]$ such that $0 = t^n_0 < t^n_1 < \cdots < t^n_n = 1$. 
	Then, let
	\b*
		\Pc_n(\mu)
		&:=&
		\Big\{
			\P \in \Pc
			~:
			B_{T_{t^n_k}} \sim^{\P} \mu_{t^n_k},
			~\forall k = 1, \cdots, n
		\Big\}.
	\e*
	The set $\Pc_n(\mu)$ consists of all Skorokhod embeddings of the $n$ marginals $(\mu_{t^n_k})_{k = 1, \cdots n}$.
	Let $\Phi_n : \Om \x (\R_+)^n \to \R$ be a reward function.
	The associated optimal SEP is formulated as
	\be \label{eq:Pn}
		P_n(\mu)
		&:=&
		\sup_{\P \in \Pc_n(\mu)}
		\E^{\P} \Big[ \Phi_n \Big( B_{\cdot\wedge T_1}, T_{t_1^n}, \cdots, T_{t_n^n} \Big) \Big].
	\ee

\subsubsection{The duality result}

	In Guo, Tan and Touzi \cite{GTT}, a duality result is established for the optimal SEP \eqref{eq:Pn}.
	Let us define
	\bqqn \label{eq:Dn}
		D_n(\mu)
		&:=&
		\inf \Big\{
			\sum_{k=1}^n \mu_{t^n_k}(\lambda_k) 
			~:(\lambda_1, \cdots, \lambda_n, H) \in (\Cc_1)^n \x \Hc~\mbox{such that}~ \nonumber \\
		&&~
			\sum_{k=1}^n \lambda_k (B_{T_{t^n_k}}) + \left(H\cdot B\right)_{T_1}
			\;\ge\;
			\Phi_n(B_{\cdot\wedge T_1}, T_{t^n_1}, \cdots, T_{t^n_n}),
			~\Pc\mbox{-q.s.}
		\Big\}.
	\eqqn

	One of the main results in \cite{GTT} is the following duality result, which is a cornerstone in our proof of Theorem \ref{theo:main}.
	
	\begin{Proposition}	\label{prop:duality_GTT}
		Suppose that Assumption \ref{ass:mu} holds true and that $\Phi_n$ is upper semi-continuous and bounded from above. Then, $P_n(\mu) = D_n(\mu)$ and the supremum of the problem $P_n(\mu)$ in \eqref{eq:Pn} is attained. 
	\end{Proposition}

\subsubsection{Optimal SEP and martingale transport problem}

	One of the main motivations for studying the optimal Skorokhod embedding problem
	is the fact that any continuous local martingale can be seen as a time changed Brownian motion.
	It is therefore natural to relate the optimal SEP to the martingale transport (MT) problem.

	Let $\Omt := C([0,1], \R)$ denote the canonical space of all continuous paths on $[0,1]$, 
	with canonical process $X$ and canonical filtration $\Ft = (\Fct_t)_{0 \le t \le 1}$.
	Let $\Mc$ denote the collection of all martingale measures on $\Om$, 
	i.e. the probability measures $\Pt$ on $(\Omt, \Fct)$ under which $X$ is a martingale.
	We recall that there is some non-decreasing $\Ft$-progressively measurable process $\langle X \rangle$ which coincides with the quadratic variation of $X$ under every martingale measure $\Pt \in \Mc$ (see e.g. Karandikar \cite{Karandikar}).
	Let
	\b*
		\langle X \rangle^{-1}_s &:=& \inf \{ t ~: \langle X \rangle_t \ge s \}.
	\e*
	Then, under every $\Pt \in \Mc$, the process $(X_{\langle X \rangle^{-1}_s})_{ s \ge 0}$ is a Brownian motion,
	and for every $t \ge 0$, $\langle X \rangle_t$ is a stopping time w.r.t. the filtration $(\Fc_{\langle X \rangle^{-1}_s})_{s \ge 0}$.
	Let $\mu = (\mu_t)_{0 \le t \le 1}$ be the given family of marginals satisfying Assumption \ref{ass:mu}.
	For $n \ge 1$ and a discrete time grid $\pi_n: 0 = t_0^n < t_1^n < \cdots < t^n_n = 1$,
	we denote
	\b*
		\Mc_n(\mu) 
		&:=&
		\big\{ \Pt \in \Mc ~: X_{t^n_k} \sim^{\Pt} \mu_{t^n_k}, ~ k = 1, \cdots, n \big\}.
	\e*
	For a reward function $\xi: \Omt \to \R$, we then define the MT problem
	\be \label{eq:Ptn}
		\tilde P_n (\mu)
		&:=&
		\sup_{\Pt \in \Mc_n(\mu)}
		\E^{\Pt} \Big[ \xi \big( X_{\cdot} \big)\Big].
	\ee
	The problem has a natural interpretation as a model-independent bound on arbitrage-free prices of the exotic option $\xi(X_{\cdot})$. 
	In order to introduce the corresponding dual formulation, let $\mathcal{\widetilde H}$ denote the collection of all $\Ft$- progressively measurable processes 
	$\tilde H:[0,1]\times\Omt\to\R$ such that $\int^\cdot\tilde H_sdX_s$ is a super-martingale under every $\Pt\in\Mc^c$. Then,
	\bqq \label{eq:dual:mot}
		\tilde D_n (\mu)
		&:=&
		\inf \Big\{
		 	\sum_{k=1}^n \mu_{t^n_k}(\lambda_k) ~~:~
			(\lambda, \widetilde H) \in (\Cc_1)^n \x \mathcal{\widetilde H}
			~\mbox{such that}~
			\nn\\
		&&~~~~~~~~~~~~~~~~~
			\sum_{k=1}^n \lambda_k(X_{t^n_k}) + \big(\widetilde H\cdot X\big)_1
			~\ge~
			\xi\left(X_{\cdot}\right),
			~\Mc\mbox{-q.s.}
		\Big\}. 
	\eqq
	The above dual problem can be interpreted as the minimal robust super-hedging cost of the exotic option, in the quasi-sure sense, using static strategies $\lambda$ and dynamic strategies $\widetilde H$.

	Via the time change argument, the above MOT problem and its dual version are related, respectively, to the optimal SEP $P_n(\mu)$ and the associated dual formulation $D_n(\mu)$. The following result is given in \cite{GTT}. It allows us to relate the limit of the above problems to our full marginal SEP; see Section 3.2 below. 
	
	\begin{Proposition} \label{prop:MT_duality}
		Suppose that Assumption \ref{ass:mu} holds true,
		and that the payoff function $\xi: \Omt \to \R$ is given by
		\be \label{eq:mtg_sep}
			\xi(X_{\cdot})
			&=&
			\Phi_n \left(
				X_{\langle X \rangle^{-1}_{\cdot}\wedge 1}, \langle X \rangle_{t_1^n}, \cdots, \langle X \rangle_{t^n_n} 
			\right),
		\ee
		for some $\Phi_n$ which is upper semi-continuous and bounded from above.
		Then, we have
		\b*
			P_n(\mu) ~~=~~ \tilde P_n(\mu)
			&=&
			\tilde D_n(\mu) ~~=~~ D_n(\mu).
		\e*
	\end{Proposition}

\subsubsection{The iterated Az\'ema-Yor embedding}

	An example of payoff function $\xi$ satisfying the conditions in Proposition \ref{prop:MT_duality}, is given by $\xi (X_{\cdot}) := \phi( X_1^*)$ with $X_1^* := \max_{0 \le t \le 1} X_t$ and $\phi: \R_+ \to \R$ a non-decreasing, bounded and c\`adl\`ag function.
	This corresponds to the function $\Phi$ defined in \eqref{eq:Phi_phi}, for which the optimal SEP, given finitely many marginals, is solved by Henry-Labord\`ere, Ob{\l}{\'o}j, Spoida and Touzi \cite{HOST}, and Ob{\l}{\'o}j and Spoida \cite{OS}.

	To solve this problem, a first technical step is to establish the following path-wise inequality (Proposition 4.1 in \cite{HOST}).
	\begin{Proposition} \label{prop:pathwise_inequality}
		Let $\xb$ be a c\`adl\`ag path on $[0,1]$ and denote $\xb^*_t := \max_{0 \le s \le t} \xb_s$.
		Then, for every $m > \xb_0$ and $\zeta_1 \le \cdots \le \zeta_n < m$:
		\be \label{eq:pathwise_inequality}
			\1_{\{\xb^*_{t_n} \ge m \}}
			&\le&
			\sum_{i = 1}^n \left( \frac{(\xb_{t_i} - \zeta_i )^+}{m - \zeta_i} + \1_{\{\xb^*_{t_{i-1}} < m \le \xb^*_{t_i} \}} \frac{m - \xb_{t_i}}{m - \zeta_i} \right) \nonumber \\
			&& - ~
			\sum_{i = 1}^{n-1} \left( \frac{(\xb_{t_i} - \zeta_{i+1} )^+}{m - \zeta_{i+1}} + \1_{\{m \le \xb^*_{t_i}, ~ \zeta_{i+1} \le \xb_{t_i} \}} \frac{\xb_{t_{i+1}} - \xb_{t_i}}{m - \zeta_{i+1}} \right).
		\ee
	\end{Proposition}

	As argued in \cite{HOST}, the above inequality implies that also the following inequality holds:
		\bqn	\label{eq:pathwise_inequality_weak}
			\1_{\{\xb^*_{t_n} \ge m \}}
			~\le~
			\sum_{i = 1}^n
			\lambda^{\zeta,m}_i(\xb_{t_i})+\int_{t_{i-1}}^{t_i}H^{\zeta,m}_t(\xb)d\xb_t,		
		\eqn
	with $T_m(\xb):=\inf\{t\ge 0: \xb_t\ge m\}$ and
	\bqq
		\lambda^{\zeta,m}_i(x)
		~:=
		\frac{(x-\zeta_i)^+}{m-\zeta_i}-\frac{(x-\zeta_{i+1})^+}{m-\zeta_{i+1}}\1_{\{i<n\}},
		~~~~~~~~~~~~~~~~~~~~~\quad x\in\R,~~~~~~~\\
			H^{\zeta,m}_t(\xb)
			~:=
			-\frac{\1_{\left(t_{i-1},t\right]}\left(T_m(\xb)\right)
			+
			\1_{\left[0,t_{i-1}\right]}\left(T_m(\xb)\right)			
			\1_{\big\{\xb_{t_{i-1}}\ge \zeta_i\big\}}}
			{m-\zeta_i},
			\;\; t\in[t_{i-1},t_i).
		\eqq 
	Indeed, if $\xb$ is continuous at $T_m(\xb)$, then the two inequalities coincide.
	If $\xb$ has a jump at $T_m(\xb)$, then the first component of the dynamic term in \eqref{eq:pathwise_inequality_weak} strictly dominates the corresponding term in \eqref{eq:pathwise_inequality}.

	Intuitively, the l.h.s. of \eqref{eq:pathwise_inequality} can be interpreted as the payoff of a specific exotic option. It serves as the basic ingredient for more general exotic payoffs since any non-decreasing function $\phi$ admits the representation $\phi(x) := \phi(0) + \int_0^x \1_{x \ge m}  d \phi(m)$.
	The r.h.s. of \eqref{eq:pathwise_inequality} can be interpreted as a model-independent super-replicating semi-static strategy, the cost of which can be computed explicitly.
	
	Minimizing the super-hedging cost yields the following optimization problem:
	\be \label{eq:Cn}
		C_n(m) 
		&:=&
		\inf_{\zeta_1 \le \cdots \le \zeta_n \le m}~
		\sum_{i=1}^n
		~\left( \frac{c_i(\zeta_i)}{m - \zeta_i} - \frac{c_i(\zeta_{i+1})}{m - \zeta_{i+1}} \1_{i < n} \right),
	\ee
	where $c_k(x) := \int_{-\infty}^x (y-x)^+ \mu_{t_k}(dy)$.
	It is argued in \cite{HOST} that the minimization problem \eqref{eq:Cn} admits at least one solution $(\hat\zeta_k(m))_{1 \le k \le n}$.
	An immediate consequence is that
	\be \label{eq:DnCn}
		D_n(\mu) ~~=~~ \tilde D_n(\mu)
		&\le&
		\phi(0) + \int_0^{\infty} C_n(m) ~d \phi(m).
	\ee
	
	Under further conditions (Assumption $\circledast$  in \cite{OS}), Ob{\l}{\'o}j and Spoida \cite{OS} provide an iterative way to solve \eqref{eq:Cn}, and to obtain a family of continuous functions $(\xi_k)_{1 \le k \le n}$ satisfying
	$\hat\zeta_k(m) = \max_{k \le i \le n} \xi_k(m), ~\forall m \ge 0$. 
	Using the family of functions $(\xi_k)_{1 \le k \le n}$, they further define a family of iterated Az\'ema-Yor embedding stopping times, given by $\tau_0 := 0$ and
	\be \label{eq:iteratedAM}
		\tau_k &:=&
		\begin{cases}
			\inf\{ t \ge \tau_{k-1} ~: B_t \le \xi_k(B^*_t)\}, &\mbox{if}~ B_{\tau_{k-1}} > \xi_k(B^*_{\tau_{k-1}}),
			\\
			\tau_{k-1}, & \mbox{else}.
		\end{cases}
	\ee
	The stopping times, $(\tau_k)_{k=1, \cdots, n}$, embed the marginals $(\mu_{t^n_k})_{k  = 1, \cdots, n}$. Moreover, it is proven in \cite{HOST} that the embedding satisfies
	\b*
		\E \big[ \phi(W^*_{\tau_n}) \big]
		&=&
		\phi(0) + \int_0^{\infty} C_n(m) d\phi(m).
	\e*
	In consequence, under Assumption $\circledast$  in \cite{OS}, it holds that
	\bqn \label{eq:MMM_n}
		P_n(\mu) ~=~ \tilde P_n(\mu)
		~=~
		\tilde D_n(\mu) ~=~ D_n(\mu)
		~=~
		\phi(0) + \int_0^{\infty} C_n(m) d\phi(m).		
	\eqn
	We notice that the discrete process $(W_{\tau_k})_{1 \le k \le n}$ resulting from this construction is in general not a Markov chain.

	\begin{Remark} \label{rem:cadlag_mot_n}
		In \cite{HOST}, the result \eqref{eq:MMM_n} is formulated for the continuous martingale problem as defined in \eqref{eq:Ptn}. However, it can be easily deduced that the solution is optimal also for the corresponding c\`adl\`ag martingale problem. 
		Specifically, let $\Omt^d$ denote the space of all c\`adl\`ag functions on $[0,1]$, $X$ the canonical space with canonical filtration $\Ft^d$, and $\Mc^d$ the space of all martingale measures.
		Define
		\b*
			\tilde P^d_n(\mu)
			\;:= 
			\sup_{\Pt \in \Mc^d_n(\mu)} \E^{\Pt} \big[ \phi(X^*_1) \big],
			\;\mbox{with }\;
			\Mc^d_n(\mu) 
			\;:=
			\Big\{
				\Pt \in \Mc^d: X_{t_k} \sim^{\Pt} \mu_{t_k},\forall k
			\Big\}.
		\e*
		It is clear that $\tilde P_n(\mu) \le \tilde P^d_n(\mu)$ since every continuous martingale is a c\`adl\`ag martingale.
		Further, by Monroe's \cite{Monroe} result, every c\`adl\`ag martingale can be represented as a time changed Brownian motion. Since $\max_{0\le t\le 1}\omega_{\theta_t}\le \omega^*_{\theta_1}$ and $\phi$ is non-decreasing, it follows that $\tilde P^d_n(\mu) \le P_n(\mu)$.
		Therefore, according to \eqref{eq:MMM_n}, for the payoff $\xi (X_{\cdot}) := \phi( X_1^*)$ with $\phi : \R \to \R$ non-decreasing,
		\bq \label{eq:max_equality}
			\tilde P_n(\mu) ~~=~~ \tilde P^d_n(\mu).
		\eq
	\end{Remark}

\subsection{The optimal SEP given full marginals}

	Our optimal SEP \eqref{eq:main_problem} given full marginals is obtained as the limit of the problem given finitely many marginals; see the proof of Theorem \ref{theo:main}.
	We provide here further discussion of the convergence of various optimal values and the corresponding optimizers. 
	
\subsubsection{The limit of MT problem given finitely many marginals}

	Our main motivation for studying the optimal SEP is the MT problem, which has a natural interpretation and applications in finance.
	For the case of finitely many marginal constraints, and for certain payoffs, the optimal SEP $P_n(\mu)$ in \eqref{eq:Pn} is equivalent to the MT problem $\tilde P_n(\mu)$ in \eqref{eq:Ptn} (cf. Proposition \ref{prop:MT_duality}).

	When the number of marginals turns to infinity, the question is whether the MT problem \eqref{eq:Ptn} converges in some sense. Specifically, we are interested in the convergence of the optimal value and of the optimizer. 
	The following convergence result is an immediate consequence of the proof of Theorem \ref{theo:main}.

	\begin{Proposition}\label{prop:limit_mot_n}
		Suppose that Assumption \ref{ass:mu} holds true and let $\xi: \Omt \to \R$ given by
			\bq
				\xi(X_{\cdot})~=~\Phi \big(X_{\langle X \rangle^{-1}_{\cdot}\wedge 1},\langle X \rangle_{1}	\big),
			\eq 
		for some upper semicontinuous and bounded $\Phi:\Omega\times\R_+\to\R$.
		Let $\tilde P_n(\mu)$ defined w.r.t. $\xi$ in \eqref{eq:Ptn} 
		and $P_n(\mu)$ and $P(\mu)$ defined w.r.t. $\Phi$. 
		Then, we have the approximation result
		\b*
			\lim_{n \to \infty} \tilde P_n(\mu) 
			~=~
			\lim_{n \to \infty} P_n(\mu) 
			~=~
			P(\mu).
		\e*
		Further, the optimal transferences converge in sense of the convergence of Skorokhod embedding (i.e. the convergence of probability measures on $\Omb$). 
	\end{Proposition}

	\begin{Remark} \label{rem:OSEP_conv}
		Recall that $\Omt$ is the canonical space of continuous functions on $[0,1]$, we define
		\b*
			\tilde P(\mu) \;:= \sup_{\Pt \in \Mc(\mu)} \E^{\Pt}\big[ \xi \big(X_{\cdot} \big) \big],
			~\mbox{with}~~
			\Mc(\mu)
			\;:=
			\Big\{
				\Pt \in \Mc : X_s \sim^{\Pt} \mu_t ~ t \in [0,1]
			\Big\}.
		\e*
		As we can see below, the limit of the optimal ($n$ marginal) continuous martingales may be a c\`adl\`ag martingale. 
		Thus the convergence of the MT problem $\tilde P_n(\mu)$ to the MT problem $\tilde P(\mu)$ fails in general.
		This underpins the importance of the full marginal SEP as the correct way of specifying the limit of the continuous $n$-marginal pricing problem.
	\end{Remark}
	

\subsubsection{The limit of the optimal martingale transference plan}
\label{subsubsec:MadanYor}

	We discuss here a specific case where the limiting martingale can be explicitly characterized. 
	Specifically, Madan and Yor \cite{MY} provide, under certain assumptions, a characterization of the continuous time martingale obtained from the Az\'ema-Yor embedding. 
	Let $b_t(x)$ be the barycenter function of $\mu_t$, defined by
	\b*
		b_t(x) &:=& \frac{\int_{[x, \infty)} y \mu_t(dy)}{\mu_t([x, \infty))} \1_{x < r_{\mu_t}} + x \1_{x \ge r_{\mu_t}}.
	\e*
	Suppose that the family $(\mu_t)_{t\in[0,1]}$ satisfies the so-called property of \emph{increasing mean residual value}: 
	\be \label{eq:MY_condition}
		t \mapsto b_t(x) ~\mbox{is non-decreasing for every}~x.
	\ee
	For any discrete time grid $\pi_n: 0 = t^n_0 < t^n_1 < \cdots < t^n_n = 1$, and under the additional Assumption $\circledast$ in \cite{OS},
	it turns out that the boundary functions $(\xi_k)_{1 \le k \le n}$ are then given by
	$\xi_k := b^{-1}$,
	and that the iterated Az\'ema-Yor embedding coincides with the Az\'ema-Yor embedding:
	\b*
		\tau_t &:=& \inf\{ s \ge 0 ~: B_s \ge \xi_t(B^*_s) \}.
	\e*
	Notice that the iterated Az\'ema-Yor embedding induces a continuous martingale,
	which is the optimal martingale transference given finitely many marginals.
	It follows that under condition \eqref{eq:MY_condition},
	this optimal martingale transference plan converges to the left-continuous right-limit martingale $M = (M_t)_{t \in [0,1]}$,
	given by
	\b*
		M_t &:=& B_{\tau_t}. 
	\e*
	In \cite{MY}, the authors prove that $M$ is in fact a Markov process and provide its generator in explicit form. 
	
	To conclude, we highlight that the Markov process $M$ defined above is a left-continuous process, a right-continuous modification gives the same generator.
	It is easily verified that
	\b*
		M^*_1 ~\le~ B^*_{\tau_1}
		&\mbox{and}&
		\P[ M^*_1 < B^*_{\tau_1}] > 0.
	\e*
	In consequence, $M$ provides no solution to the MT given full marginals; see also Remark \ref{rem:OSEP_conv}.

\subsubsection{The limit of the pathwise inequality}

	The proof of Theorem \ref{theo:optimser:dual} is based on applying limiting arguments to the path-wise inequality \eqref{eq:pathwise_inequality} (cf. Section \ref{sec:proof_theo:optimizer:dual}). By use of a similar argument, we might obtain an almost sure inequality for c\`adl\`ag martingales. 
	
	\begin{Proposition}\label{prop:limit_ineq}
	 	Let $M$ a right-continuous martingale, $\zeta:[0,1)\times\R_+\to\R$ such that $\zeta^m_\cdot\in\V^+_l$ and $\zeta^m_\cdot<m$, and $\phi:\R_+\to\R_+$ bounded, continuous and non-decreasing.   		
		Suppose either i) that $\zeta$, $\phi$ and the marginals of $M$ satisfy the conditions of Theorem \ref{theo:optimser:dual}; or ii) that $\zeta$ admits the representation \eqref{eq:assumption_countable}.
		Then, with $\lambda^\zeta$ given in \eqref{eq:def_lambda}, $M$ satisfies the following inequality: 
		\bqn \label{eq:mtg_ineq_prop}
			\phi(M^*_1)
			~\le~
			\int_0^1\lambda^\zeta (ds,M_s)
			+\int_0^1\int_0^\infty 
			\1_{\left\{m\le M^*_{t^-};\zeta^m_t\le M_{t^-}\right\}}\frac{d\phi(m)}{m-\zeta^m_t}dM_t,
			\quad a.s.
		\eqn
	\end{Proposition}
	
		The difference between the r.h.s. of \eqref{eq:mtg_ineq_prop} and \eqref{eq:superhedge} appear in the dynamic terms (cf. \eqref{eq:dyn_strat_n}). Specifically, for the martingale formulation, the counterpart of the first dynamic term in \eqref{eq:superhedge} is always negative and thus vanishes from the inequality. This is related to the fact that the limit of the first dynamic component in \eqref{eq:pathwise_inequality_weak} is zero. For continuous martingales, the two inequalities coincide. 
	

\subsection{The resolution of $C(m)$} 
	
	Finally, we would like to discuss the resolution of the problem $C(m)$ in \eqref{eq:maxzeta},
	since the main results in Theorems \ref{them:main2} and \ref{theo:optimser:dual}
	rely on its solution $\hat\zeta$.
	
	First, it is clear that we can decompose the minimization problem $C(m)$ as follows:
	\b*
		C(m) ~= \inf_{x < m} ~ \left\{ \frac{c(0, x)}{m - x} + v(0, x)\right\};
		~~v(0, x) ~:= \inf_{\zeta \in \tilde \V^+_l, \zeta_0 = x} \int_0^1 \frac{\partial_t c(s,\zeta_s)}{m - \zeta_s} ds.
	\e*
	
	The problem to compute $v(0,x)$ is a standard singular deterministic control problem.
	When the function $\partial_t c(s, x)$ is continuous, it therefore follows by standard arguments (see e.g. \cite{crandall1992user}) that $v$ can be characterized as a viscosity solution
	to the PDE
	\be \label{eq:PDE_Cm}
		\max \left\{
			-~ \partial_x {  v(t,x )} , 
			~ -~ \partial_t v(t,x) ~-~ \frac{\partial_t c(t,x)}{m - x} 
			\right\}
			&=&
			0,
	\ee
	equipped with the terminal condition $v(1,x) = 0$, for all $x < m$.
	
	{ 
	We now propose a numerical scheme for the problem $C(m)$.
	To this end, for a given partition $\pi_n=\{t^n_1,...,t^n_n\}$, with $0=t^n_0\le ...\le t^n_n=1$, let $\V^n_l$ the subset of $\V^+_l$ for which $\zeta$ is constant on $(t_{i-1},t_i]$, $i=1,...,n$. Further, let 	
		\bq
			v^n(0, x) 
			~:= \inf_{\substack{\zeta \in \V^n_l, \\ \zeta_0 = x,~\zeta_1< m}} \int_0^1 \frac{\partial_t c(s,\zeta_s)}{m - \zeta_s} ds.
		\eq
	For a sequence of partitions such that $|\pi_n|\to 0$, it follows that $v^n(0,x)\to v(0,x)$; cf. the proof of Lemma \ref{lemm:C_limit} below. 
	On the other hand,
		\bq
			v^n(0, x)	~=
			\inf_{\substack{\zeta_i,~i=1,...,n:\\ x\le \zeta_1\le ...\le \zeta_n<m}}
			\sum_{i=1}^{n}
			\frac{\Delta c(t^n_{i}, \zeta_{i})}{m - \zeta_{i}},
		\eq
	with $\Delta c(t^n_i,\zeta):=c(t^n_{i}, \zeta)-c(t^n_{i-1}, \zeta)$. 
	In consequence, $v^n(0,x)~=~\bar v^n(t^n_0,x)$, where $\bar v^n(t^n_k,x)$, $k=0,...,n$, is iteratively defined by
		\bq\left\{\begin{array}{lll}
			\bar v^n(t^n_k,x)&=&
			\inf_{0\le y<m-x}
			\left(
			\bar v^n(t^n_{k+1},x+y)
			~+~\frac{\Delta c(t^n_i,x+y)}{m - (x+y)}
			\right),
			~~ 
			k\le n-1,\\
			\bar v^n(t^n_n,x) &=&0.
		\end{array}\right.\eq 
	This yields a scheme for explicit calculation of $v^n(0,x)$ as an approximation of $v(0,x)$. 
	}


\section{Proofs}
\label{sec:Proofs}

\subsection{Technical lemmas}

		\begin{Lemma}\label{lem:approx_fcn} 
		Let $\Phi:\Omb\to\R$ be upper semicontinuous, bounded and such that $\Phi(\om, \theta) = \Phi(\om_{\cdot\wedge\theta_1}, \theta)$, for all $(\om,\theta) \in \Omb$. 
		Further, let  $(\pi_n)_{n \ge 1}$ be a sequence of discrete time grids with $|\pi_n| \to 0$. 
		Then, there exists a sequence $(\Phi_n)_{n\ge 1}$ of upper semicontinuous and bounded functions $\Phi_n:\Om\times(\R_+)^n\to \R$, such that 
		\bqn	\label{eq:Phi_n_prop}
			\Phi_n(\omega_{\cdot\wedge\theta_1}, \theta_{t^n_1},...,\theta_{t^n_n})
			\searrow
			\Phi(\omega_{\cdot\wedge\theta_1}, \theta_{\cdot}),
			~~\mbox{as}~ n \to \infty,
			\quad \forall (\theta,\omega)\in\Omb.
		\eqn
	\end{Lemma}
	
	\begin{proof}
		Let $\Phi^k:\Om\times\V^+_r\to\R$, $k\in\N$, such that $\Phi^k(\om, \theta) = \Phi^k(\om_{\cdot\wedge\theta_1}, \theta)$, $\Phi^k$ is bounded and Lipschitz and $\Phi^k\searrow \Phi$. Then, let $\Phi^{k,m}:\Om\times\V^+_r\to\R$, $m\in\mathbb{N}$, be given by
		\bq
			\Phi^{k,m}(\omega,\theta):=\Phi^n(\omega,\bar\theta^m), 		
		\eq
		where $\bar\theta^m:=\theta_{t_i}$, $s\in\left[\frac{i}{m},\frac{i+1}{m}\right)$, $i=1,...,m$, for $\theta\in\V^+_r$.
		Note that since $d(\theta,\bar\theta^m)\le \frac{1}{m}$ (cf. \eqref{eq:levy_metric}), we have
			\bq
				\left|\Phi^{k,m}(\omega,\theta)-\Phi^k(\omega,\theta)\right|
				~=~
				\left|\Phi^{k}(\omega,\bar\theta^m)-\Phi^k(\omega,\theta)\right|
				~\le ~\frac{L_k}{m},				
			\eq 
		with $L_k$ the Lipschitz constant associated with $\Phi^k$. In consequence, 
			\bq
				\hat\Phi^{k,m}(\omega,\theta):=\Phi^{k,m} (\omega,\theta)+\frac{L_k}{m}\searrow \Phi^k(\omega,\theta),
				\qquad m\to\infty. 
			\eq
		Hence, we may choose $m_k$ such that
				$\hat\Phi^{k,m_k}(\omega,\theta)	\searrow \Phi(\omega,\theta)$, $k\to\infty$.
		In consequence, defining $\Phi_n:=\hat\Phi^{k_n,m_{k_n}}$, with $k_n=\max\{k\in\N: m_k<n\}$, we have that $\Phi_n$, $n\in\mathbb{N}$, satisfy \eqref{eq:Phi_n_prop} and we conclude. 
	\qed
	\end{proof}

	\vspace{3mm}

	\noindent {\bf Proof of Lemma \ref{lemm:Zeta}.}
	We follow the argument at the beginning of Section 3 of \cite{HOST}.
	Let
	\be \label{eq:phi}
		\Psi_m(\zeta)
		&:=&
		\frac{c(0, \zeta_0)}{m - \zeta_0}
		~+~
		\int_0^1 \frac{\partial_t c(s, \zeta_s)}{m - \zeta_s} ds.
	\ee
	We first consider a constant function $\hat \zeta^z_\cdot \equiv z$ for some constant $z < m$.
	By direct computation, it is easy to see that
	\b*
		\Psi_m(\hat \zeta^z) ~=~ \frac{c(1, z )}{m - z}  ~~\ge~~ C(m).
	\e*
	Note that since $z \mapsto c(1,z)$ is convex and $\frac{c(1, z )}{m - z}$ is the slope of the tangent to $z \mapsto c(1,z)$ intersecting the $x$-axis in $m$,
	it follows that $C(m) < 1$.
	
	On the other hand, since $\partial_t c(s, z) \ge 0$, we have
	\b*
		\Psi_m(\zeta) &\ge& \frac{c(0, \zeta_0)}{m - \zeta_0} ~\to~ 1
		~~\mbox{as}~~ \zeta_0 \to -\infty.
	\e*
	For the minimization problem $C(m)$ in \eqref{eq:maxzeta}, it is therefore enough to consider the space
	{  $\V^+_l([0,1), [K, m))$} for some constant $K \in (-\infty, m)$, i.e.
	\b*
		C(m) &=& \inf_{  \zeta \in \V^+_l([0,1), [K,m))} \Psi_m(\zeta).
	\e*
	Notice that $\zeta \mapsto \Psi_m(\zeta)$ is continuous and {  $\V^+_l([0,1), [K,m))$} is compact under the L\'evy metric.
	It follows that, for every $m>0$, there exists at least one solution in $\V^+_l$ to \eqref{eq:maxzeta}.
	
	To conclude, it is enough to use a measurable selection argument to choose a measurable function $\hat\zeta$.
	\qed

	\vspace{3mm}

	\begin{Lemma}\label{lemm:C_limit}
		Recall that $C_n(m)$ is defined by \eqref{eq:Cn}.
		Suppose that the function $c$ is differentiable in $t$ and that the derivative function $\partial_t c$ is continuous. 
		Then, for every $m>0$, we have	
		$$\lim_{n\to\infty}C_n(m) ~~=~~ C(m).$$
	\end{Lemma}
	
	\begin{proof}
		{  Let $\V^n_l$ the subset of $\V^+_l$ for which $\zeta$ is constant on $(t_{i-1},t_i]$, $i=1,...,n-1$, and on $(t_{n-1},t_n)$. For $n$ fixed and $\zeta\in\V^n_l$, let $\zeta(t_n):=\zeta(t_n^-)$.
		Notice that for every $\zeta\in\V^n_l$,
		\bqq
			\Phi(\zeta)
		&=& 
			\frac{c(0, \zeta_{t_0})}{m - \zeta_{t_0}}
			~+~
			\sum_{i=1}^{n}
				\int_{t_{i-1}}^{t_{i}} \frac{\partial_t c(s, \zeta_{t_{i}})}{m - \zeta_{t_{i}}} ds\\
		&=&  \frac{c(0, \zeta_{t_0})}{m - \zeta_{t_0}}
			~+~
			\sum_{i=1}^{n}
			\left(\frac{c(t_{i}, \zeta_{t_{i}})}{m - \zeta_{t_{i}}}
			-
			\frac{c(t_{i-1}, \zeta_{t_{i}})}{m - \zeta_{t_{i}}}\right)\\
		&=&
			\sum_{i=0}^n\left(\frac{c(t_{i}, \zeta_{t_{i}})}{m - \zeta_{t_{i}}}
			-
			\frac{c(t_{i}, \zeta_{t_{i+1}})}{m - \zeta_{t_{i+1}}}
			\1_{i<n}\right).
		\eqq
	Since $x_0\le m$, it holds that $\frac{c(t_{0}, \zeta_{t_{0}})}{m - \zeta_{t_{0}}}
			-
			\frac{c(t_{0}, \zeta_{t_{1}})}{m - \zeta_{t_{1}}}\ge 0$. In consequence, 
		\bq 
			\inf_{\zeta\in\V^n_l, \zeta_1\le m}\Phi(\zeta)
			~=~
			\inf_{\zeta\in\V^n_l, \zeta_1\le m}
			\sum_{i=1}^n\left(\frac{c(t_{i}, \zeta_{t_{i}})}{m - \zeta_{t_{i}}}
			-
			\frac{c(t_{i}, \zeta_{t_{i+1}})}{m - \zeta_{t_{i+1}}}
			\1_{i<n}\right)
			~=~ C_n(m). 
		\eq }		
	Hence, the $C_n(m)$ are non-increasing in $n$ and
		\be
			C_n(m)
			~=~
			\inf_{\zeta\in\V^n_l,\zeta_1\le m}\Phi(\zeta)
			&\ge&
			\inf_{\zeta\in\V^+_l,\zeta_1\le m}\Phi(\zeta)
			~=~
			C(m).
		\ee 
			
	Next, for any $\zeta \in \V^+_l$, by direct truncation, we can easily obtain a sequence $\zeta^n$ such that
	$\zeta^n \in \V^n_l$ and $\zeta^n \to \zeta$ under the L\'evy metric.
	It follows that $C_n(m) \to C(m)$ as $n \to \infty$.
	\qed
	\end{proof}	
	
	\begin{Lemma} \label{lem:cont}
		The mapping from $\Omb$ to $\R$ given by:
		\bqn
			\big(\om, \theta \big)
			~\longmapsto~
			\om^* \big( \theta(1) \big)
			~:=~
			\sup_{0\le s\le \theta(1)} \om (s),
			\label{eq:continuity_max}
		\eqn
		is continuous with respect to the product topology on $\Omb$.
	\end{Lemma}
	\begin{proof}
	Let $\big(\omega^n(\cdot),\theta^n(\cdot)\big)\in\Omb$, $n\in\N$, converging in the product topology to $\big(\tilde \omega(\cdot),\tilde \theta(\cdot)\big)\in\Omb$. 
	Recall that $C(\R_+, \R)$ is equipped with the metric $\rho$ defined in \eqref{eq:rho}, which induces the topology of uniform convergence on compact subsets. Hence, 
		\bqn
			\lim_{n\to\infty}\sup_{0\le s\le m}
			\big|\omega^n(s)-\tilde \omega(s)\big|=0,\quad \textrm{for all $m\ge 0$}.\label{eq:metric:whitt}
		\eqn
	Further, convergence in the L\'evy metric is equivalent to point-wise convergence at each point of continuity. Due to the right-continuity of elements in $\V^+_r([0,1], \R_+)$ and the extended definition of the L\'evy metric (cf. \eqref{eq:levy_metric}), it follows that $\theta^n(1)$ converges to $\tilde \theta(1)$. 	 
	Note that
	\bqq
		\Big| \sup_{0\le s\le \theta^n(1)} \omega^n(s)-\sup_{0\le s\le \tilde \theta(1)} \tilde \omega(s)\Big|
		&=&
		\;\;\Big| \sup_{0\le s\le \theta^n(1)} \omega^n(s)-\sup_{0\le s\le \theta^n(1)} \tilde \omega(s)\Big|\\
		&&+~
		\Big| \sup_{0\le s\le \theta^n(1)} \tilde \omega(s)-\sup_{0\le s\le \tilde \theta(1)} \tilde \omega(s)\Big|.
	\eqq
	The first term is dominated by $\sup_{0\le s\le \theta^n(1)} \left|\omega^n(s)-\tilde \omega(s)\right|$ which tends to zero as $n$ tends to infinity due to \eqref{eq:metric:whitt}. 
	Since $\tilde\omega(\cdot)$ is a continuous path, also the second term tends to zero. Hence, the mapping in \eqref{eq:continuity_max} is continuous and we conclude. 
	\qed
	\end{proof}

	{ 	

	\begin{Lemma} \label{lemm:conv_static}
		Let $\Omt^d := D([0,1], \R)$ be the space of all c\`adl\`ag paths on $[0,1]$ with canonical process $X$, 
		and $\Mc^d$ be the space of all martingale measures on $\Omt^d$. We define 
		\bq 
			\Mc^d(\mu)~:=~\left\{ \Pt \in \Mc^d~: X_t \sim^{\Pt} \mu_t,~\forall t\in[0,1] \right\}.
		\eq
		Further, let $\zeta: [0,1) \to (-\infty, m)$ be a non-decreasing c\`agl\`ad path on $[0,1)$ 
		and $\pi_n :0 = t^n_0 < \cdots < t^n_n = 1$ be a sequence of discrete time grids such that $|\pi_n| \to 0$ as $n \to \infty$. 
		Let $\zeta^c$ the continuous part of $\zeta$ and let $\lambda^{\zeta,m}$ defined in \eqref{eq:lambda_m}. 
		Then,
			\begin{itemize}
				\item[i)]{if $\mu_t$ is atomless, $t\in[0,1]$, we have $\Mc^d(\mu)\mathrm{-q.s.}$,
		\be \label{eq:static_conv_org}
			\sum_{k=1}^{n}
			\left(
				\frac{(X_{t^n_k} - \zeta_{t^n_k})^+}{m - \zeta_{t^n_k}}
				-
				\frac{(X_{t^n_k} - \zeta_{t^n_{k+1}})^+}{m - \zeta_{t^n_{k+1}}}\1_{\{k<n\}}
			\right)			
			\longrightarrow
			\int_0^1\lambda^{\zeta,m}\left(X_t,dt\right).\nn
		\ee
		}
		\item[ii)]{if $\zeta_s=\sum_{k=0}^\infty \zeta_k\1_{(t_k,t_{k+1}]}(s)$, then the convergence in i) holds path-wise for all $\xb \in \Omt^d$. The integral with respect to $d\zeta^c$ is then identically zero.}
		\end{itemize}
	\end{Lemma}

	\begin{proof}
		It follows from the definition of $\lambda^{\zeta,m}$, that in order to prove i), it is sufficient to show that, $\Mc^d(\mu)\mathrm{-q.s.}$, 
		{\setlength{\arraycolsep}{-1.55cm}
		\bqqn	\label{eq:static_conv}
			\phantom{hejap}\sum_{k=1}^{n-1}
			\left(
				\frac{(X_{t^n_k} - \zeta_{t^n_{k+1}})^+}{m - \zeta_{t^n_{k+1}}}
				-
				\frac{(X_{t^n_k} - \zeta_{t^n_k})^+}{m - \zeta_{t^n_k}}
			\right)\\
			&&\longrightarrow
			\int_0^1 \frac{X_t - m}{(m-\zeta_t)^2} \1_{X_t \ge \zeta_t} ~d\zeta^c_t
			\;+\;
			\sum_{t}
			\left(
				\frac{(X_t - \zeta_{t+})^+}{m - \zeta_{t+}}
				-
				\frac{(X_t- \zeta_{t})^+}{m - \zeta_{t}}
			\right).\nn 
		\eqqn}
		Observe that for each path $\xb \in D([0,1], \R)$, the discrete sum in \eqref{eq:static_conv} might be written as $\int_0^1f_n(t;\zeta)d\zeta_t$, where  
			\bq 
				f_n(t;\zeta)
				~=~
				\sum_{k=1}^{n-1}
			\frac{\1_{t\in (t^n_{k},t^n_{k+1}]}}{\zeta_{t^n_{k+1}}-\zeta_{t^n_{k}}}
			\left(
				\frac{(\xb_{t^n_k} - \zeta_{t^n_{k+1}})^+}{m - \zeta_{t^n_{k+1}}}
				-
				\frac{(\xb_{t^n_k} - \zeta_{t^n_k})^+}{m - \zeta_{t^n_k}}
			\right).
			\eq	
		Denote by $D_{\zeta} \subset (0,1)$ the subset of all discontinuous points of $\zeta$. First, suppose that assumption i) holds. 
		Then, for $t\not\in D_\zeta^c\cap\{t:\xb_t=\zeta_t\}$, the $f_n(\cdot;\zeta)$ converges point-wise to $f(\cdot;\zeta)$, with
	\begin{equation}
	f(t;\zeta)~=~
	\left\{\begin{array}{lll}
		\frac{\xb_t - m}{(m-\zeta_t)^2} \1_{\xb_t \ge \zeta_t},  
	&& t \in D_\zeta^c \\
		\frac{1}{\zeta_{t^+}-\zeta_t}\left(\frac{(\xb_t - \zeta_{t+})^+}{m - \zeta_{t+}}
				~-~
				\frac{(\xb_t- \zeta_{t})^+}{m - \zeta_{t}}\right),
	&& t \in D_\zeta.
	\end{array}\right.\end{equation}
	On the other hand, by use of Fubini's theorem and assumption i), we obtain that for $\P\in\Mc^d(\mu)$,  
		\bq
			\E^\P\left[\int_0^1\1_{\{\xb_t=\zeta_t\}}d\zeta^c_t\right]
			~=~\int_0^1\P\left[\xb_t=\zeta_t\right]d\zeta^c_t
			~=~\int_0^1\mu_t\left(\{\zeta_t\}\right)d\zeta^c_t
			~=~ 0.
		\eq
	That is to say, $\int_0^1\1_{\{X_t=\xi_t\}}d\zeta^c_t=0$, $\Mc^d(\mu)\mathrm{-q.s.}$ 
	Since, for all $\varepsilon >0$, $\zeta \to \frac{(x-\zeta)^+}{m-\zeta}$ is Lipschitz on $(-\infty,m-\varepsilon]$, there is $K>0$, such that $f_n(t)\le K$, $t\in[0,1]$, $n\ge 0$. Hence, by use of dominated convergence we obtain $\int_0^1f_n(t;\zeta)d\zeta_t\to\int_0^1 f(t;\zeta)d\zeta_t$, $\Mc^d(\mu)\mathrm{-q.s.}$, which implies \eqref{eq:static_conv}.
	
	Next, suppose assumption ii) holds. Then, for $t\in[0,1]$, the function $f_n(t;\zeta)$ converges point-wise to $f^0(t;\zeta)$, where
	\begin{equation}
	f^0(t;\zeta)~=~
	\left\{\begin{array}{lll}
		0,  
	&& t \in D_\zeta^c \\
		\frac{1}{\zeta_{t^+}-\zeta_t}\left(\frac{(\xb_t - \zeta_{t+})^+}{m - \zeta_{t+}}
				~-~
				\frac{(\xb_t- \zeta_{t})^+}{m - \zeta_{t}}\right),
	&& t \in D_\zeta.
	\end{array}\right.\end{equation}
	By use of the same arguments as in the case i), we may then apply the dominated convergence theorem pathwise and we easily conclude.
	\qed
	\end{proof}		
}

\subsection{Proof of Theorem \ref{theo:main}}

	We first argue that the optimal SEP given finitely many marginals defined by \eqref{eq:Pn}, may be reformulated similarly to that in \eqref{eq:alpha} and \eqref{eq:P_embedding}.	 
	Concretely, for a given discrete time grid $\pi_n: 0 = t_0^n < \cdots < t_n^n=1$, we call a $(\mu, \pi_n)$-embedding a term
	\be \label{eq:alpha_n}
		\alpha 
		&=& 
		\big(\Om^{\alpha}, \Fc^{\alpha}, \P^{\alpha}, \F^{\alpha} = (\Fc_t^{\alpha})_{t \ge 0}, (W^{\alpha}_t)_{t \ge 0}, (T^{\alpha}_k)_{k = 1, \cdots, n} \big),
	\ee
	such that in the filtered space $\big(\Om^{\alpha}, \Fc^{\alpha}, \P^{\alpha}, \F^{\alpha} \big)$,
	$W^{\alpha}_{\cdot}$ is a Brownian motion, 
	$T^{\alpha}_1 \le \cdots \le T^{\alpha}_n $ are all stopping times,
	the stopped process $(W^{\alpha}_{T^{\alpha}_n \wedge \cdot})$ is uniformly integrable,
	and $W^{\alpha}_{T^{\alpha}_k} \sim^{\P^{\alpha}} \mu_{t_k^n}$ for each $k = 1, \cdots, n$.
	Let $\Ac_n(\mu)$ denote the collection of all $(\mu, \pi_n)$-embeddings $\alpha$.
	Then it is clear that every term in $\Ac_n(\mu)$ induces on the canonical space $\Omb$ a probability measure in $\Pc_n(\mu)$,
	and every probability measure $\P \in \Pc_n(\mu)$ together with the space $(\Omb, \Fcb, \Fbb)$ forms a $(\mu, \pi_n)$-term in $ \Ac_n(\mu)$.
	And hence, for $\Phi_n: \Om \x (\R_+)^n \to \R$ a given reward function, we have that
	\be	\label{eq:Pn_embedding}
		P_n(\mu)
		&=&
		\sup_{\alpha \in \Ac_n(\mu)}
		\E^{\P^{\alpha}} \Big[ \Phi_n \Big( W^{\alpha}_{\cdot}, T^{\alpha}_1, \cdots, T^{\alpha}_n \Big) \Big].
	\ee

	\vspace{3mm}
	
	Before proving Theorem \ref{theo:main}, we present a Lemma. Its proof is partly adapted from the proof of Theorem 11 in Monroe \cite{Monroe} and that of Theorem 3.10 in Jakubowski \cite{Jakubowski}.	
	
	\begin{Lemma}\label{lemm:tight}
		Let $\alpha_n\in\Ac_n(\mu)$, $n\in\N$, be a sequence of terms of the form \eqref{eq:alpha_n}.
		Let $\P_n$ be the probability measure on $\Omb$ induced by $ \big(W^{\alpha_n}_{\cdot}, T^{\alpha_n}_{\cdot} \big)$ in the probability space $ \big(\Om^{\alpha_n}, \Fc^{\alpha_n}, \P^{\alpha_n}\big)$.
		Then, the sequence $\left\{\P_n\right\}_{n \ge 1}$ is tight, and any limiting point $\P$ is in $\Pc(\mu)$.	 
	\end{Lemma}
	

	\begin{proof} 
		\noindent \rmi We first claim that the sequence $\left\{\P_n\right\}_{n \ge 1}$ is tight.
	Indeed, the projection measure $\P_n |_{\Om}$ on $\Om$ is the Wiener measure for every $n \ge 1$,
	and hence the sequence $(\P_n|_{\Om})_{n \ge 1}$ is trivially tight.
	Next, since $T^{\alpha_n}_1$ are all minimal stopping times in the sense of Monroe \cite{Monroe},
	it follows from Proposition 7 in \cite{Monroe} that
		\b*
			\P_n \left(T_1 \ge \lambda\right)
			&=&
			\P^{\alpha_n} \left(T^{\alpha_n}_1 \ge \lambda\right)
			~\le~ 
			\lambda^{-1/3}\left(\E^{\P^{\alpha_n} }\big[|X^{\alpha_n}_1|\big]^2+1\right),
			\quad \forall \lambda>0.
		\e*
	Let  $A_\lambda$ be the set of functions in $\V^+_r([0,1], \R_+)$ which are bounded by $\lambda>0$ and $\pi$ the projection of $\Omb$ onto $\V^+_r([0,1], \R_+)$, it follows that 
		\bqn
			\P_n(\pi^{-1}(A_\lambda))
			~~=~~
			\P_n \left(T_1 \le \lambda\right)
			~~\ge~~ 
			\lambda^{-1/3}\left( \left( \mu_1(|x|) \right)^2+1\right).\label{eq:tight}		
		\eqn
	Since $A_\lambda$, $\lambda>0$, are compact and the r.h.s. of \eqref{eq:tight} can be made arbitrarily small by an appropriate choice of $\lambda$, the sequence of projection measures $(\P_n|_{\V^+_r})_{n \ge 1}$ is also tight. 
	In consequence, $\left\{\P_n\right\}_{n \ge 1}$ is tight. 

	\vspace{1mm}
	
	\noindent 	\rmii Let $\P$ be a limit point of $(\P_n)_{n \ge 1}$,
	by taking subsequences if necessary,
	we can assume that $\P_n \to \P$.
	We now prove that $B$ is a $\Fbb$-Brownian motion under the limit measure $\P$.

	Since the measures $\P_n$ are induced by $(W^{\alpha_n}, T^{\alpha_n})$ under $\P^{\alpha_n}$,
	we know that $B$ is a $\Fbb$-Brownian motion under each $\P_n$, $n \ge 1$.
	Let $t > s$, $0 < \eps < t-s$, and $\phi: \Omb \to \R$ be a bounded continuous function which is $\Fcb_{s+\eps}$-measurable,
	then for every $\varphi \in C_b^2(\R)$, we have
	\b*
		\E^{\P_n} 
		\left[ \phi(B_{\cdot}, T_{\cdot}) \left( \varphi(B_t) - \varphi(B_{s+\eps}) - \int_{s+\eps}^t \frac{1}{2} \varphi''(B_u) du \right) \right] = 0.
	\e*
	By taking the limit $n \to \infty$, it follows that
	\be \label{eq:BrownianMartPb}
		\E^{\P} 
		\left[ \phi(B_{\cdot}, T_{\cdot}) \left( \varphi(B_t) - \varphi(B_{s+\eps}) - \int_{s+\eps}^t \frac{1}{2} \varphi''(B_u) du \right) \right] = 0.
	\ee
	According to Lemma \ref{lemm:BorelTribu}, the equality \eqref{eq:BrownianMartPb} holds true also for every bounded random variable $\phi: \Omb \to \R$ that is $\Fcb_t$-measurable.
	Let $\eps \to 0$, it follows that for every $\phi: \Omb \to \R$ bounded and $\Fcb_t$-measurable,
	and every $\varphi \in C_b^2(\R)$ that
	\b*
		\E^{\P} 
		\left[ \phi(B_{\cdot}, T_{\cdot}) \left( \varphi(B_t) - \varphi(B_{s}) - \int_{s}^t \frac{1}{2} \varphi''(B_u) du \right) \right] = 0.
	\e*
	And hence $B$ is a $\Fbb$-Brownian motion under $\P$.
	
	\vspace{1mm}
	
	\noindent \rmiii Now, we show that the process $(B_{t \wedge T_1} )_{t \ge 0}$ is uniformly integrable under $\P$.
	For every $\eps > 0$, there is $K_{\eps} > 0$ such that
	\b*
		\int_{R} \big( |x| - K_{\eps} \big)^+ \mu_1(dx) 
		&\le&
		\eps.
	\e*
	Since $|x| \1_{\{|x| \ge 2K\}} \le 2 (|x| - K)^+$, it follows that
	\b*
		\E^{\P_n} \Big[ \big|B_{T_1 \wedge t} \big| \1_{\{|B_{T_1 \wedge t} | \ge K_{\eps}\}} \Big]
		&\le&
		2 \E^{\P_n} \Big[ \Big( \Big| B_{T_1} \Big| - K_{\eps} \Big)^+ \Big]
		~~\le~~
		2\eps,
		~~~ \forall t \ge 0.
	\e*
	Then, for every bounded continuous function $p: \R \to \R$ such that $p(x) \le |x| \1_{\{|x| \ge 2 K_{\eps}\}}$,
	it follows by the dominated convergence theorem that
	\b*
		\E^{\P} \Big[ p \big( B_{t \wedge T_1} \big) \Big]
		&=&
		\lim_{n \to \infty} \E^{\P_n} \Big[ p \big( B_{t \wedge T_1} \big) \Big]
		~~\le~~
		2 \eps, ~~~\forall t \ge 0,
	\e*
	which implies that $(B_{t \wedge T_1})_{t \ge 0}$ is uniformly integrable under $\P$.

	\vspace{1mm}

	\noindent \rmiv Next, we prove that $B_{T_t} \sim^{\P} \mu_t$, $t \in [0,1]$. 
	We shall adapt the idea of proof of Theorem 3.10 in Jakubowski \cite{Jakubowski}.
	For every $\alpha_n$ the solution of the optimal SEP \eqref{eq:Pn},
	denote by $m^{\alpha_n}$ the random measure on $([0,1], \Bc([0,1]))$ defined by
	\b*
		m^{\alpha_n}([0,t], \om)
		&:=&
		\frac{T^{\alpha_n}_t(\om)}{1 + T^{\alpha_n}_1(\om)},
		~~\forall t \in [0,1].
	\e*
	Since $m_n$ takes value in a compact space (the space of all positive measures on $[0,1]$ with mass less than $1$), we know that the sequence of distribution of $m_n$ under $\P^{\alpha_n}$ is tight.
	By taking subsequences and the Skorokhod representation theorem, we can assume that there is some probability space $(\Om^*, \Fc^*, \P^*)$ in which
	\b*
		\big( W^{\alpha_n}, T^{\alpha_n}_1, m^{\alpha_n} \big)
		&\rightarrow&
		\big(W^*, T^*, m^* \big),
		~~\P^*-a.s.
	\e*
	
	Further, the map $t \mapsto \E^{\P^*} \big[ m^*([0,t]) \big]$ from $[0,1]$ to $\R$ is non-decreasing, and hence admits at most countable discontinuous points.
	It follows that there is some countable set $\Q_1 \subset [0,1)$ such that $\E^{\P^*} \big[ m^*(\{t\}) \big] = 0$, for every $t \in [0,1] \setminus \Q_1$.
	Thus, for every $t \in [0,1] \setminus \Q_1$, we have  $\P^*$-a.s.,
	$m^{\alpha_n}([0,t]) \to m^*([0,t])$, and hence $T^{\alpha_n}_t \to T^*_t$.
	In particular, we have
	\b*
		W^{\alpha_n}_{T^{\alpha_n}_t} 
		&\to&
		W^*_{T^*_t},
		~~\P^*-a.s.
		~\forall t \in [0,1] \setminus \Q_1.
	\e*	
	Besides, by Hirsch and Roynette \cite{HR} Lemma 4.1., there exists a countable set $\Q_2 \subset [0,1]$ such that $t \mapsto \mu_t$ is continuous at any $s \in [0,1] \setminus \Q_2$.
	Then for every $t \in [0,1] \setminus (\Q_1 \cup \Q_2)$, we have
	\be \label{eq:marginal_BTt}
		\P \circ (B_{T_t})^{-1} = \P^* \circ (W^*_{T^*_t})^{-1} = \mu_t.
	\ee
	By the right continuity of $t \mapsto \mu_t$, it follows that \eqref{eq:marginal_BTt} holds true for every $t \in [0,1]$.
	
	\vspace{1mm}
	
	In summary, we have proven that in the filtered space $(\Omb, \Fcb_{\infty}, \P, \Fbb)$,
	$B$ is a Brownian motion, $T_1$ is a minimal stopping time 
	and $B_{T_t} \sim \mu_t$ for every $t \in [0,1]$.
	We easily conclude. 
	\qed
	\end{proof}	
	
	\vspace{3mm}	
	
	\noindent {\bf Proof of Theorem \ref{theo:main}}. 
	By taking expectation over each side of the inequality defining $\Dc(\mu)$ in \eqref{eq:Dc_mu}, for all $\P \in \Pc(\mu)$, we easily obtain the weak duality $P(\mu) \le D(\mu)$.
	Let $\Phi_n$, $n\in\mathbb{N}$, the sequence of functions approximating $\Phi$ as given in Lemma \ref{lem:approx_fcn}. Further, let $P_n(\mu)$ and $D_n(\mu)$ the primal and dual $n$-marginal problems defined w.r.t. $\Phi_n$ in \eqref{eq:Pn} and \eqref{eq:Dn}. Since $\Phi_n \ge \Phi$, it follows that
	\bqn	\label{eq:Dn_D_proof}
		D_n(\mu) ~\ge~ D(\mu).
	\eqn

	For each $n\in\N$, let $\alpha_n\in\Ac_n(\mu)$ be the solution of the optimal SEP $P_n(\mu)$ in \eqref{eq:Pn_embedding} defined with respect to $\Phi_n$.
	Let $\P_n$ be the probability measure on $\Omb$ induced by $ \big(W^{\alpha_n}_{\cdot}, T^{\alpha_n}_{\cdot} \big)$ in the probability space $ \big(\Om^{\alpha_n}, \Fc^{\alpha_n}, \P^{\alpha_n}\big)$.
	Then, according to Lemma \ref{lemm:tight}, the sequence $\left\{\P_n\right\}_{n \ge 1}$ is tight, and $\P\in\Pc(\mu)$, where $\P$ is a limiting point of $\left\{\P_n\right\}_{n \ge 1}$. 
	
	Note that by taking sub-sequences if necessary, we can assume that $\P_n \to \P$. 	
	By use of the monotone convergence theorem and the optimality of the $\P_n$ for the problem \eqref{eq:Pn}, $n\in\N$, it follows that
	\bqq \label{eq:Pncvg}
		P (\mu)
		&\ge&
		\E^{\P} \Big[ \Phi \big( B_{\cdot}, T_{\cdot} \big) \Big]
		~=~
		\lim_{n \to \infty} \E^{\P} \Big[ \Phi_n \big( B_{\cdot}, T_{\cdot} \big) \Big] \nonumber \\
		&\ge&
		\lim_{n \to \infty} \left( \lim_{k \to \infty} \E^{\P_k} \Big[ \Phi_n \big( B_{\cdot}, T_{\cdot} \big) \Big] \right)
		~\ge~
		\lim_{n \to \infty} \left( \lim_{k \to \infty} \E^{\P_k} \Big[ \Phi_k \big( B_{\cdot}, T_{\cdot} \big) \Big] \right) \nonumber \\
		&=&
		\lim_{n\to\infty} P_n(\mu).
	\eqq
		In consequence, since $P_n(\mu) \ge P(\mu)$ for all $n \ge 1$, we have that
		\bqn	\label{eq:limitPn_proof}
			\lim_{n\to\infty}P_n(\mu)~=~ P(\mu). 
		\eqn	
		Since the $\Phi_n$ satisfy \eqref{eq:Phi_n_prop}, we may apply the duality result for the optimal SEP with finitely many marginal constraints (see Proposition \ref{prop:duality_GTT}). Hence, it follows from \eqref{eq:limitPn_proof} combined with \eqref{eq:Dn_D_proof} that 
		$P(\mu) ~\ge~ D(\mu).$
	Combined with the weak duality, this yields $P(\mu) = D(\mu)$. 	
	As a by-product,  we also obtain that $\P$ is an optimal embedding for the optimal SEP \eqref{eq:main_problem}. This concludes the proof. 
	\qed

	\vspace{3mm}

	\noindent {\bf Proof of Proposition \ref{prop:limit_mot_n}.}
		Given the form of $\xi: \Omt \to \R$, Proposition \ref{prop:MT_duality} applies. Hence, $\tilde P_n(\mu)=P_n(\mu)$.
		Next, note that $P_n(\mu)$ is of the form \eqref{eq:Pn}, for all $n\in\N$. Let $\P_n$ be the optimal measure for $P_n(\mu)$. 
		Then, according to Lemma \ref{lemm:tight}, passing to a subsequence if necessary, $\P_n \to \P$, with $\P\in\Pc(\mu)$. It follows that
	\b* 
		P (\mu)
		~\ge~
		\E^{\P} \Big[ \Phi \big( B_{\cdot}, T_{\cdot} \big) \Big]
		~\ge~
		\lim_{n \to \infty} \E^{\P_n} \Big[ \Phi \big( B_{\cdot}, T_{\cdot} \big) \Big]
		~=~
		\lim_{n\to\infty} P_n(\mu).
	\e*
		Since $P_n(\mu)\ge P(\mu)$, for $n\ge 1$, we easily conclude. 
	\qed

\subsection{Proof of Theorem \ref{them:main2}}

	By Lemma \ref{lem:cont}, the mapping
	\bq
		\big(\omega(\cdot),\theta(\cdot)\big)\longmapsto
		\sup_{0\le s\le \theta(1)} \omega(s),
	\eq
	is continuous with respect to the product topology on $\Omb$. 
	Hence, Theorem \ref{theo:main} applies and $\lim_{n\to\infty} P_n(\mu) = P(\mu)=D(\mu)$.
	
	Next, for the optimal SEP with finitely many marginals, we have (see \eqref{eq:DnCn})
	\b*
		P_n(\mu) ~~=~~ D_n(\mu) 
		&\le& \phi(0) + \int_0^{\infty} C_n(m) d \phi(m), 
	\e*
	where the equality holds under Assumption \ref{ass:A} \rmii (see \eqref{eq:MMM_n}).
	
	Finally, it is enough to use Lemma \ref{lemm:C_limit} together with the monotone convergence theorem to deduce that
	\bqq
		~~~~~~~~~~~~~~~~~~~~~~~~	
		\lim_{n\to\infty}\int_0^\infty C_n(m)\;d\phi(m)
		&=&
		\int_0^\infty C(m)\;d\phi(m).	
		~~~~~~~~~~~~~~~~~~~~~~\qed	
	\eqq

\subsection{Proof of Theorem \ref{theo:optimser:dual}}\label{sec:proof_theo:optimizer:dual}

	Due to assumption \eqref{eq:assump_zeta}, the pair $(\hat\lambda,\widehat H)$ is well-defined and $(\hat\lambda,\widehat H)\in\Lambda(\mu)\times\Hc$.
		According to \eqref{eq:quest:price} and \eqref{eq:lambda_m}, for $m>0$ and $\zeta\in\V^+_l$ such that $\zeta<m$, the cost of $\lambda^{\zeta,m}$ is given by
			{\setlength{\arraycolsep}{-2.775cm}
			\bqq \label{eq:mu:chain}
			\mu\big(\lambda^{\zeta,m}\big)
			\;=\;
				\frac{c(1,\zeta_1)}{m - \zeta_1}
				-\sum_{t\in D}\left[\frac{c(t,\zeta_{t^+})}{m-\zeta_{t^+}}
			-\frac{c(t,\zeta_t)}{m-\zeta_t}\right]
			-\int_0^1
				\frac{\partial}{\partial \zeta}
				\left.\left\{\frac{c(t,\zeta)}{m-\zeta}\right\}\right|_{\zeta=\zeta^{c}_t}d\zeta^{c}_t\nn\\
			&&=\;	\frac{c(0, \zeta_0)}{m - \zeta_0}
			~+~
				\int_0^1 \frac{\partial_t c(s, \zeta_s)}{m - \zeta_s} ds,~~~~~~~~~~~
				~~~~~~~~~~~~~~~~~~~~~~~~~~~~~~~~~~~~~~~~~~
			\eqq }
	where it was used that $\frac{\partial}{\partial \zeta}\frac{(x-\zeta)^+}{m-\zeta}
	=\1_{\{x\ge \zeta\}}\frac{x-m}{(m-\zeta)^2}$. 
	Since $\hat\zeta^m_\cdot$ minimizes \eqref{eq:maxzeta}, it follows that $\mu\big(\lambda^{\hat\zeta^m,m}\big) = C(m)$.  
	Integration w.r.t. $d\phi(m)$ and application of Theorem \ref{them:main2} and Corollary \ref{cor:main}, yields $\mu(\hat\lambda) = D_0(\mu)$.

	Next, let $\pi_n :0 = t^n_0 < \cdots < t^n_n = 1$ be a sequence of discrete time grids such that $|\pi_n| \to 0$, as $n \to \infty$. According to Proposition \ref{prop:pathwise_inequality}, for $(\theta,\omega)\in\Omb$,
		\bqqn \label{eq:pathwise_inequality3}
			\1_{\left\{\omega^*_{\theta(1)}\ge m \right\}}
		&\le&
			\sum_{i = 1}^{n} \left(\frac{\big(\omega_{\theta(t_i)} - \hat\zeta^m_{t_i} \big)^+}{m - \hat\zeta^m_{t_i}}
			-\frac{\big(\omega_{\theta(t_i)} - \hat\zeta^m_{t_{i+1}} \big)^+}{m - \hat\zeta^m_{t_{i+1}}}\1_{\{i<n\}}\right)\nn\\
		&&+ \sum_{i = 1}^n\1_{\{\omega^*_{\theta(t_{i-1})} \;<\; m \;\le\; \omega^*_{\theta(t_i)} \}} \frac{m - \omega_{\theta(t_i)}}{m - \hat\zeta^m_{t_i}} \nn\\
		&&- \sum_{i = 1}^{n-1}\1_{\{m \;\le\; \omega^*_{\theta(t_i)}; ~ \hat\zeta^m_{t_{i+1}} \;\le\; \omega_{\theta(t_i)} \}} \frac{\omega_{\theta(t_{i+1})} - \omega_{\theta(t_i)}}{m - \hat\zeta^m_{t_{i+1}}}.
		\eqqn
	Note that the l.h.s. does not depend on the partition. Hence, in order to verify that $(\hat\lambda,\widehat H)$ satisfies \eqref{eq:superhedge}, it suffices to show that, for all $\P\in\Pc(\mu)$, integrated w.r.t. $d\phi(m)$ the r.h.s. in \eqref{eq:pathwise_inequality3} converges $\P$-a.s. to $\int_0^1\hat\lambda(B_{T_s},ds) + \int_0^{T_1} \widehat H_s dB_s$.
		
	Application of Lemma \ref{lemm:conv_static} with $X_t=B_{T(t)}$ and $\zeta=\hat\zeta^m$, yields that the static term in \eqref{eq:pathwise_inequality3} converges to 
		\bq
			\int_0^1\lambda^{\hat\zeta^m,m}\left(B_{T(t)},dt\right),
			\quad \Pc(\mu)\mathrm{-q.s.}
		\eq
	In consequence, integrated w.r.t. $d\phi(m)$, the static term in \eqref{eq:pathwise_inequality3} converges to $\hat\lambda(\overline B)=\int_0^1\hat\lambda(B_{T_s},ds)$, $\Pc(\mu)\mathrm{-q.s.}$

	As for the first dynamic term in \eqref{eq:pathwise_inequality3}, the definition of $\tau_m$ yields:
		\bqq
			\sum_{i = 1}^n\1_{\{\omega^*_{\theta(t_{i-1})} < m \le \omega^*_{\theta(t_i)} \}} \frac{m-\omega_{\theta(t_i)}}{m - \hat\zeta^m_{t_i}}
			~\longrightarrow~
			\int_0^{\theta(1)} \frac{\1_{[\tau_m,I^+(\tau_m)]}(s)}{m-\hat\zeta^m_{\theta^{-1}(\tau_m)}}dB_s,
			\quad \Pc-\mbox{q.s.}
		\eqq
	Integration with respect to $d\phi$, then gives the convergence of the corresponding terms. 

	To prove the convergence of the second dynamic term in \eqref{eq:pathwise_inequality3}, we first integrate with respect to $d\phi$. The integrated term may then be re-written as follows: 
		{\setlength{\arraycolsep}{-2cm}
		\bqqn	\label{eq:proof_conv_dyn}
			\sum_{i = 1}^{n-1}
			\int_0^\infty
			\1_{\big\{m \;\le\; \omega^*_{\theta(t_i)};~\hat\zeta^m_{t_{i+1}} \;\le\; \omega_{\theta(t_i)}\big\}} \frac{d\phi(m)}{m - \hat\zeta^m_{t_{i+1}}}\big(\omega_{\theta(t_{i+1})} - \omega_{\theta(t_i)}\big)	\nn\\
			&&=
			\sum_{i = 1}^{n-1}
			\int_0^{\hat\zeta^{-1}_{t_{i+1}}\left(\omega_{\theta(t_i)}\right)\wedge \omega^*_{\theta(t_i)}}
			\frac{d\phi(m)}{m - \hat\zeta^m_{t_{i+1}}}
			\int_{\theta(t_i)}^{\theta(t_{i+1})}dB_s\nn\\
			&&=
			\int_0^{\theta(1)}
			\sum_{i = 1}^{n-1}
			\int_0^{\hat\zeta^{-1}_{t_{i+1}}\left(\omega_{\theta(t_i)}\right)\wedge \omega^*_{\theta(t_i)}}
			\frac{d\phi(m)}{m - \hat\zeta^m_{t_{i+1}}}
			\1_{\big(\theta(t_i),\theta(t_{i+1})\big]}(s)dB_s,
		\eqqn
		}
	where $\hat\zeta^{-1}_{t_{i+1}}$ denotes the inverse of $\hat\zeta^{\cdot}_{t_{i+1}}$.
	We denote the integrand in \eqref{eq:proof_conv_dyn} by $H^n$. Note that the $H^n$ are predictable. 
	Further, since $\hat\zeta$ satisfies \eqref{eq:assump_zeta} and $\phi$ is bounded, the $H^n$ are uniformly bounded. 
	Due to assumption ii), we also have that $H_n\to H$ on $\mathrm{\Omega}\times(0,\infty)$, with
			\bq
			H_s~=~
			\int_0^{\hat\zeta^{-1}_{\theta^{-1}(s)}\left(\omega_{I^-(s)}\right)\wedge \omega^*_{I^-(s)}}
			\frac{d\phi(m)}{m - \hat\zeta^m_{\theta^{-1}(s)}}.		
		\eq 		
	In consequence, for all $\P\in\Pc$, we have that $\int_0^{\theta(1)}H^n_sdB_s \to\int_0^{\theta(1)}H_sdB_s$ in probability. Hence, convergence holds a.s. along a subsequence and we conclude. 	
	\qed
	
	\vspace{3mm}
	
	
		\noindent {\bf Proof of Proposition \ref{prop:limit_ineq}.}
		Let $m>0$ fixed and $\phi(x)=\1_{\{x\ge m\}}$; the general case follows by integration with respect to $d\phi(m)$. 
		Recall that \eqref{eq:pathwise_inequality} holds for all c\`adl\`ag paths $\xb$ (in the proof of Theorem \ref{theo:optimser:dual} we only made use of this result for continuous paths). 
		For a given sequence of partitions $\pi_n :0 = t^n_0 < \cdots < t^n_n = 1$, such that $|\pi_n| \to 0$ as $n \to \infty$, it therefore suffices to argue that the r.h.s. in \eqref{eq:pathwise_inequality} converges to the r.h.s. of \eqref{eq:mtg_ineq_prop}. 
		The static term in \eqref{eq:mtg_ineq_prop} coincides with the static term in \eqref{eq:superhedge}. Hence, the convergence follows by use of the same arguments as in the proof of Theorem \ref{theo:optimser:dual}. 
		Next, consider the first dynamic term in \eqref{eq:pathwise_inequality}. For each c\`adl\`ag path $\xb \in D([0,1], \R)$, let $\tau^m(\xb):=\inf\{t\ge 0: \xb_t\ge m\}$. Due to the right-continuity of $\xb$, we have $\xb_{\tau^m(\xb)}\ge m$. In consequence, 
	\bq
		\lim_{n\to\infty}\sum_{i = 1}^n \1_{\{\xb^*_{t_{i-1}} < m \le \xb^*_{t_i} \}} \frac{m - \xb_{t_i}}{m - \zeta_{t_i}}
		=  \frac{m - \xb_{\tau^m(\xb)}}{m - \zeta_{\tau^m(\xb)}}
		\le 0.
	\eq
	Next, consider the second dynamic term in \eqref{eq:pathwise_inequality}. First, we argue its convergence under assumption ii). To this end, we rewrite it as follows: 
	{\setlength{\arraycolsep}{-0.7cm}
				\bqq \label{eq:proof_conv_dynamic}
					\sum_{i = 1}^{n-1}
			\1_{\big\{m \le \xb^*_{t_i};\;\zeta_{t_{i+1}} \le \xb_{t_i}\big\}}
			\frac{\xb_{t_{i+1}} - \xb_{t_i}}{m - \zeta_{t_{i+1}}}\\
					&&=~\sum_{i = 1}^{n-1}
			\1_{\big\{m \le \xb^*_{t_i};\;\zeta_{t_{i+1}} \le \xb_{t_i}\big\}}
			\int_{t_i}^{t_{i+1}}\frac{dM_t}{m - \zeta_{t_{i+1}}}\nn\\
					&&=~
			\int_0^1
			\sum_{i = 1}^{n-1}
			\1_{\big\{m \le \xb^*_{t_i};\;\zeta_{t_{i+1}} \le \xb_{t_i}\big\}}
			\frac{\1_{\left(t_i,t_{i+1}\right]}(t)}{m - \zeta_{t_{i+1}}}
			\;dM_t.
	\eqq}
		Denote the integrand on the r.h.s. by $H^n$. Note that the $H^n$ are predictable and uniformly bounded. Further, since $\zeta_s=\sum_{k=0}^\infty \zeta^m_k\1_{(t_k,t_{k+1}]}(s)$, we may choose a sequence of partitions $\pi_n :0 = t^n_0 < \cdots < t^n_n = 1$ such that $H_n\to H$ on $D([0,1], \R)$, with 
		\bq
			H_t~=~
			\1_{\big\{m \le \xb^*_{t-};\zeta_t \le \xb_{t-}\big\}}
			\frac{1}{m - \zeta_t}.
		\eq
	It follows that $\int_0^1H^n_tdM_t\to\int_0^1H_tdM_t$ in probability; cf. Theorem I.4.40 in \cite{jacod1987limit}. Hence, convergence holds a.s. along a subsequence and we conclude. 
	Under assumption i), the result follows by, first, integrating the pathwise inequality w.r.t. $d\phi(m)$ and, then, modifying the argument along the same lines as in the proof of Theorem \ref{theo:optimser:dual} (cf. \eqref{eq:proof_conv_dyn}). 	
	\qed 
	
	\subsection{Proof of Proposition \ref{prop:mtq:new}}	
		W.l.o.g., let $m>0$ fixed and $\phi(x)=\1_{\{x\ge m\}}$. The general case then follows by integration with respect to $d\phi(m)$.
		The proof is based on the path-wise inequality \eqref{eq:pathwise_inequality}. Since $\zeta_\cdot$ is non-decreasing and $M$ is c\`adl\`ag, it implies that 	
	{\setlength{\arraycolsep}{-2.15cm}
	\bqq \label{eq:mtg_ineq_n}
		\1_{\left\{M^*_1\ge m\right\}}~\le~
		\sum_{i=1}^n
		~\bigg( \frac{\big(M_{t_i} - \zeta_{t_{i}}\big)^+}{m - \zeta_{t_{i}}} - \frac{(M_{t_i} - \zeta_{t_{i+1}})^+}{m - \zeta_{t_{i+1}}} \1_{i < n} \bigg)&\\
			&&-	
			\sum_{i = 1}^{n-1} \1_{\{m \le M^*_{t_i}, ~ \zeta_{t_{i+1}} \le M_{t_i} \}} \frac{M_{t_{i+1}} - M_{t_i}}{m - \zeta_{t_{i+1}}}
			+
			\sum_{i = 1}^n \1_{\{M^*_{t_{i-1}} < m \le M^*_{t_i} \}} \frac{m - M_{t_i}}{m - \zeta_{t_i}}.\nn
	\eqq}
	We proceed by taking expectation on both sides of this inequality and, then, passing to the limit. To this end, note that the expected value of the dynamic terms is bounded from above by zero (cf. Proposition 3.2 in \cite{HOST}).
	Since the l.h.s. of the inequality is independent of the partition, it follows that 
	\bqq 
		\E \Big[ \1_{\left\{M^*_1\ge m\right\}} \Big]
		&\le&
		\frac{\E\left[(M_{1} - \zeta_{1-})^+\right]}{m - \zeta_{1-}}
		- 
		\lim_{n\to\infty}
		\int_0^1
		f_n(t)
		d\zeta_t,
	\eqq
	where	
	\bqq 
		f_n(t)~=~
		\sum_{i=1}^{n-1}
		~\frac{\1_{t\in (t_{i},t_{i+1}]}}{\zeta_{t_{i+1}}-\zeta_{t_{i}}}
		\bigg( \frac{\E\left[(M_{t_i} - \zeta_{t_{i+1}})^+\right]}{m - \zeta_{t_{i+1}}}
		- \frac{\E\left[(M_{t_i} - \zeta_{t_{i}})^+\right]}{m - \zeta_{t_{i}}}\bigg).
	\eqq	
	Let 
	\begin{equation}
	f(t)~:=~
	\left\{\begin{array}{lll}
		\frac{\E\left[\left(M_t-\zeta_t\right)^+\right] - \P\left[M_t>\zeta_t\right](m-\zeta_t)}{(m-\zeta_t)^2},  
	&& t \in D_\zeta^c \\
		\frac{1}{\zeta_{t^+}-\zeta_t}\left(\frac{\E\left[(M_t - \zeta_{t^+})^+\right]}{m - \zeta_{t^+}}-\frac{\E\left[(M_t - \zeta_t)^+\right]}{m - \zeta_t}\right),
	&& t \in D_\zeta.
	\end{array}\right.\end{equation}
	
	Observe that since $M_t$ is integrable, $\zeta\to \E\left[(M_{t} - \zeta)^+\right]/(m - \zeta)$ is Lipschitz on $(-\infty,m-\varepsilon]$, for all $\varepsilon >0$. Hence, it is differentiable almost everywhere and it follows that $f_n(t)$ converges point-wise to $f(t)$, for $t\in[0,1]\setminus D$, where $D=D_\zeta^c\cap\{t:F(\cdot,t)\textrm{ discontinuous at $\zeta_t$}\}$, with $F(x;t):=\P\left[M_t>x\right]$. 
	Moreover, there is $K>0$ such that $|f_n(t)|\le K$, $t\in[0,1]$, $n> 0$.
	Hence, by use of dominated convergence, it follows that $\int_0^1 f_n(t)d\zeta_t \to \int_0^1 f(t)d\zeta_t$ and we conclude. 
	\qed 

\appendix
\section{Appendix}

	We provide a characterization of the $\sigma$-field on the canonical space $\Omb := \Om \x \V_r$.

	\begin{Lemma} \label{lemm:BorelTribu}
		The Borel $\sigma$-field of the Polish space $\Omb$ is given by 
		$\Fcb_{\infty} := \bigvee_{t \ge 0} \Fcb_t$.
		Moreover, $\Fcb_{t-} := \bigvee_{0 \le s < t} \Fcb_s$ 
		coincides with  the $\sigma$-field generated by all bounded continuous functions 
		$\xi : \Omb \to \R$ which are $\Fcb_t$-measurable. 
	\end{Lemma}
	\proof \rmi We first prove that $\Fcb_{\infty}$ is the Borel $\sigma$-field of the Polish space $\Omb$.
	Define $\Vc^+_t$ as the $\sigma$-field on $\V^+_r$, generated by all sets of the form $\{\theta \in \V^+_r, \theta_u \le s\}$ for $u \in [0,1]$ and $s \le t$; and $\Vc^+_{\infty} := \cup_{t \ge 0} \Vc^+_t$.
	Then $\Fcb_{\infty} = \Fc^0_{\infty} \otimes \Vc^+_{\infty}$, where
	$\Fc^0_{\infty} := \cup_{t \ge 0} \Fc_t^0$ is the Borel $\sigma$-field of $\Om$ (see .e.g. the discussion at the beginning of Section 1.3 of Stroock and Varadhan \cite{SV}).
	So it is enough to check that $\Vc^+_{\infty}$ is the Borel $\sigma$-field $\Bc(\V^+_r)$ of the Polish space $\V^+_r$.
	First, by the right-continuity of $\theta \in \V^+_r$, 
	the L\'evy metric on $\V^+_r$ can be defined equivalently by
	\b*
		d(\theta,\theta')
		~:=~
		\inf \Big\{\eps>0 ~: 
		\theta_{t-\eps} - \eps
		~\le~ 
		\theta'_t
		~\le~
		\theta_{t+\eps}+\eps,
		~\forall t \in \Q \cap [0,1] \Big\},	
	\e*
	where $\Q$ is the collection of all rational numbers.
	Then it follows that $\Bc(\V^+_r) \subseteq \sigma \big( T_u ~: u \in \Q \big) \subseteq \Vc^+_{\infty}$.
	On the other hand, for every $u \in [0,1)$,
	the map $ \theta \mapsto \frac{1}{\eps} \int_u^{u+\eps} \theta(s) ds$ is continuous under L\'evy metric and hence Borel measurable.
	Letting $\eps \to 0$, it follows that $\theta \mapsto \theta(u)$ is also Borel measurable,
	and hence $\Vc^+_{\infty} \subseteq \Bc(\V^+_r)$.
	We then obtain that $\Fcb_{\infty} = \Bc(\Omb)$.
	
	\vspace{1mm}
	
	\noindent \rmii We now consider the $\sigma$-field generated by bounded continuous functions.
	First, it is well known that the filtration $\F^0$ on $\Om$ is left-continuous and
	$\Fc^0_{t-} = \Fc^0_t$ is generated by all bounded continuous functions $\xi_1 : \Om \to \R$ which are $\Fc^0_t$ continuous.

	Next, we notice that for every $t \ge 0$,
	\b*
		\Vc^+_{t-} 
		&:=&
		\bigvee_{0 \le s < t} \Vc^+_s 
		~~:=~~ 
		\sigma \big( T_u \wedge t ~: u \in [0,1] \big).
	\e*
	Let $\xi_2 : \V^+_r \to \R$ be a bounded continuous function which is also $\Vc^+_t$-measurable.
	Then $\xi_2 \big( (\theta_u)_{u \in [0,1]} \big) = \xi_2 \big( (\theta_u \wedge u)_{r \in [0,1]} \big)$,
	which is $\sigma \big( T_u \wedge t ~: u \in [0,1] \big)$-measurable.
	On the other hand, the function 
	$ \theta \mapsto \frac{1}{\eps} \int_u^{u+\eps}  (T_{\ell}(\theta) \wedge t) d \ell$
	from $\V^+_r$ to $\R$ is continuous and $\Vc^+_t$-measurable for $\eps >0$.
	The by taking $\eps \to 0$, it follows that $T_u \wedge t$ is measurable w.r.t.
	the $\sigma$-field generated by all bounded continuous functions $\xi_2 : \V^+_r \to \R$ which are $\Vc^+_t$-measurable.
	Therefore, $\Vc^+_{t-}$ is the $\sigma$-field generated by all bounded continuous functions on $\V^+_r$ which are $\Vc^+_t$-measurable.
	
	Finally, since  $\Fcb_{s} = \Fc^0_{s} \otimes \Vc^+_{s}$,
	it follows that $\Fcb_{t-} = \Fc^0_{t-} \otimes \Vc^+_{t-}$.
	We hence conclude that $\Fcb_{t-}$ is  the $\sigma$-field generated by all bounded continuous functions 
		$\xi : \Omb \to \R$ which are $\Fcb_t$-measurable.
	\qed

\end{document}